\documentclass[12pt,a4paper,reqno]{amsart}
\usepackage{amsmath}
\usepackage{amsthm}
\usepackage{amsfonts}
\usepackage{amssymb}
\usepackage{color}
\usepackage{mathrsfs}
\usepackage{graphicx}

\numberwithin{equation}{section}

\theoremstyle{plain}
\newtheorem{theorem}{Theorem}

\theoremstyle{plain}
\newtheorem*{rfr}{Reverse flow recipe}

\theoremstyle{plain}
\newtheorem*{irl}{Irrational rotation lemma}

\theoremstyle{plain}
\newtheorem*{question}{Question}

\theoremstyle{plain}
\newtheorem*{assertiona}{Assertion~A}

\theoremstyle{plain}
\newtheorem*{assertionb}{Assertion~B}

\theoremstyle{plain}
\newtheorem{lemma}{Lemma}[section]

\theoremstyle{remark}
\newtheorem*{remark}{Remark}

\theoremstyle{definition}
\newtheorem{example}{Example}[section]

\theoremstyle{definition}
\newtheorem*{convention}{Convention}

\theoremstyle{definition}
\newtheorem*{step1}{Step~1}

\theoremstyle{definition}
\newtheorem*{step2}{Step~2}

\theoremstyle{definition}
\newtheorem*{step3}{Step~3}

\theoremstyle{definition}
\newtheorem*{step4}{Step~4}

\theoremstyle{definition}
\newtheorem*{step5}{Step~5}

\theoremstyle{definition}
\newtheorem*{case1}{Case~1}

\theoremstyle{definition}
\newtheorem*{case2}{Case~2}

\theoremstyle{definition}
\newtheorem*{case3}{Case~3}

\theoremstyle{definition}
\newtheorem*{case4}{Case~4}

\theoremstyle{definition}
\newtheorem*{case5}{Case~5}

\theoremstyle{definition}
\newtheorem*{case6}{Case~6}

\def\T{\mathbf{T}}

\def\dd{\mathrm{d}}

\def\eps{\varepsilon}

\def\Nn{\mathbb{N}}

\def\BBB{\mathcal{B}}

\def\EEE{\mathcal{E}}
\def\FFF{\mathcal{F}}

\def\HHH{\mathcal{H}}

\def\LLL{\mathcal{L}}

\def\NNN{\mathcal{N}}

\def\PPP{\mathcal{P}}
\def\QQQ{\mathcal{Q}}
\def\RRR{\mathcal{R}}

\def\VVV{\mathcal{V}}
\def\WWW{\mathcal{W}}

\def\frakM{\mathfrak{M}}

\def\frakR{\mathfrak{R}}

\def\AAAA{\mathscr{A}}

\def\IIII{\mathscr{I}}

\def\MMMM{\mathscr{M}}

\def\SSSS{\mathscr{S}}

\renewcommand{\le}{\leqslant}
\renewcommand{\ge}{\geqslant}

\title{Irreversible and dissipative systems}

\author[Beck]{J. Beck}

\address{Department of Mathematics, Hill Center for the Mathematical Sciences, Rutgers University, Piscataway NJ 08854, USA}

\email{jbeck@math.rutgers.edu}

\author[Chen]{W.W.L. Chen}

\address{School of Mathematical and Physical Sciences, Faculty of Science and Engineering, Macquarie University, Sydney NSW 2109, Australia}

\email{william.chen@mq.edu.au}

\author[Yang]{Y. Yang}

\address{School of Science, Beijing University of Posts and Telecommunications, Beijing 100876, China}

\email{yangyx@bupt.edu.cn}

\begin{document}



\begin{abstract}
We study some new dynamical systems where the corresponding piecewise linear flow is
neither time reversible nor measure preserving.
We create a dissipative system by starting with a finite polysquare translation surface,
and then modifying it by including a one-sided barrier on a common vertical edge of two adjacent atomic squares,
in the form of a union of finitely many intervals.
The line flow in this system partitions the system into a transient set and a recurrent set.
We are interested in the geometry of these two sets.
\end{abstract}

\maketitle

\thispagestyle{empty}

%
%

\section{Introduction}\label{sec1}

We continue our research into non-integrable systems concerning billiard flow in polygons and solids
as well as the related problem of geodesic flow on flat surfaces and in flat manifolds.
In the cases that we have studied thus far, the flow is piecewise linear, time reversible and measure preserving.
In particular, the last property permits the application of the Birkhoff ergodic theorem, and this leads to uniformity
of the relevant orbits inside flow-invariant sets.

In this paper, we study some new dynamical systems where the corresponding piecewise linear flow is
neither time reversible nor measure preserving.
Here the work of Veech~\cite{veech69} concerning the $2$-circle problem serves as our motivation.
Related to the $2$-circle problem is a polysquare translation surface comprising $2$ atomic squares,
modified by the introduction of some partial barriers on some edges of the squares.
The question then arises as to what may happen if these barriers are one-sided, permitting flow in one direction
but not the other.
We show that these one-sided barriers are responsible for the violation of
both the time reversible and the measure preserving properties.

Dynamical systems with such long term behaviour are said to be \textit{dissipative}.
Intuitively, it means that the time evolution of the flow contracts the space-volume.

The \textit{Lorentz attractor}, motivated by weather prediction and claiming that the time evolution contracts
the space-volume exponentially fast, is perhaps famous and notorious with equal measure.
It is the quintessential mathematical model of \textit{chaos theory} which has a huge literature.
Thus it is somewhat unsatisfactory that little has been proved with any mathematical rigour concerning the Lorentz attractor.

Our flat dissipative systems here are quite different, and exhibit slow contraction of the space-volume,
so we may use the term \textit{slow chaos}.
However, we are able to carry out a rigorous mathematical discussion on these systems with precise proofs.

Time irreversibility is a very interesting property that brings one to the well known \textit{Loschmidt paradox},
or reversibility-irreversibility paradox, which serves as a criticism of the so-called H-theorem of Boltzmann.
Here a fundamental problem of the \textit{proof} of Boltzmann is \textit{circular reasoning} and concerns an element of time asymmetry
artifically injected into an argument to establish time asymmetry.

The Loschmidt paradox is still fresh after $150$ years, perhaps due to the lack of any time irreversible
mathematical model with a detailed theory providing theorems and proofs.
We hope our work here will bring a little insight into this very intriguing question.

We introduce a dissipative system $\PPP$ by starting first with a finite polysquare translation surface,
and then modifying it by including a one-sided barrier $B$ on a common vertical edge of two adjacent atomic squares,
in the form of a union of finitely many vertical intervals, and where a set $A$ corresponding to $B$ is chosen
on a different vertical edge on the same horizontal street that contains~$B$.
We consider a flow moving from left to right.
When a flow line with irrational slope hits the barrier $B$ from the left, it continues from the corresponding point of the set~$A$.
We then observe that there are points on $\PPP$ that stay away from the flow line, apart possibly from the early stages,
and there are points on $\PPP$ with arbitrarily small neighbourhoods that are visited by the flow line infinitely often
as time goes to infinity.
The system $\PPP$ is thus divided into two subsets, and we are interested in the structure of these two subsets
which we call respectively the \textit{transient} set and the \textit{recurrent} set.

We begin our investigation by considering in Section~\ref{sec2} the special case of the $2$-square torus,
modified by the inclusion of a one-sided barrier on the common vertical edge.
We show that there is a simple reverse flow recipe which helps us investigate this question,
and that both the transient set and the recurrent set are finite unions of polygons with total area $1$
and with boundary edges that are either vertical or parallel to the direction of the flow.
We also explain how this problem also possesses a form of \textit{cyclic symmetry}.

We then proceed in Section~\ref{sec3} to investigate the corresponding problem with the $n$-square torus
modified in a similar way.
Here we may lose the cyclic symmetry which is crucial in the earlier investigation.
Nevertheless we are able to establish some partial results.
Then in Section~\ref{sec4}, we consider the more general problem of starting with a polysquare translation surface,
and use simple surfaces related to the L-surface to illustrate some properties of the transient sets and the recurrent sets.
We show that each can take up a very large proportion of the system or a very small proportion of the system,
and that the recurrent set can be split into disjoint components.
We then make a thorough study of the case of the modified L-surface.

In Section~\ref{sec5}, we study the problem of modified polysquare translation surfaces in greater detail.
Here we develop an extension process which is iterative and helps us study more closely the recurrent set.
We also use some important results of Kakutani and Kac.
However, the main question is whether this extension process is finite, in which case we can conclude that
both the transient set and the recurrent set are finite unions of polygons
and with boundary edges that are either vertical or parallel to the direction of the flow.

Section~\ref{sec6} is devoted to the study of this question.
We begin this investigation by studying two special cases of the problem,
one by making some assumptions on the one-sided barrier and the other by making some assumption on the recurrent set.
We conclude our investigation by finally showing that the extension process is indeed always finite,
and this leads to the very nice description of the transient sets and the recurrent sets.

%
%

\section{Modification of the Veech model}\label{sec2}

Consider a modified \textit{billiard} table comprising $2$ atomic squares, with a one-sided partial barrier of length~$b$, where $0<b<1$,
on the common edge separating the atomic squares, as shown in Figure~1.
This barrier is one-sided in the sense that when the billiard hits it from the left, as shown in the picture on the left,
it rebounds like normal billiard, whereas when the billiard hits it from the right, as shown in the picture on the right,
it passes straight through.
For simplicity, we have put this one-sided barrier against the bottom edge, but it can be anywhere on the vertical
edge separating the atomic squares.

\begin{displaymath}
\begin{array}{c}
\includegraphics[scale=0.8]{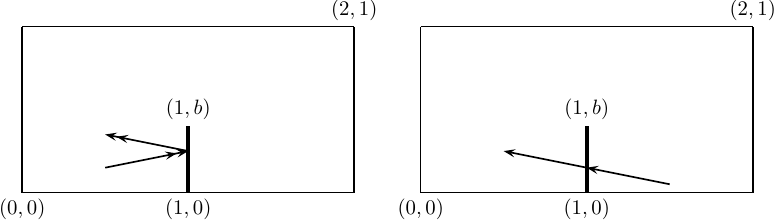}
\\
\mbox{Figure 1: a modified billiard table}
\end{array}
\end{displaymath}

Note that the billiard flow in Figure~1, if not entirely horizontal or vertical, is a $4$-direction flow.
We can convert it into a $1$-direction flow on some modified \textit{system}
using the simple idea of K\"{o}nig and Sz\"{u}cs~\cite{KS13} in 1913 that involves reflection across a horizontal
axis and across a vertical axis.
This results in $1$-direction flow on the modified system as shown in Figure~2.
However, we need to describe this $1$-direction flow very carefully.

\begin{displaymath}
\begin{array}{c}
\includegraphics[scale=0.8]{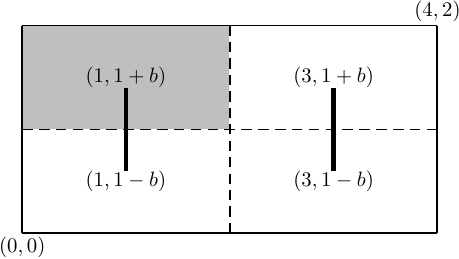}
\\
\mbox{Figure 2: a modified system corresponding to Figure~1}
\end{array}
\end{displaymath}

Figure~3 describes the flow on the modified system.
The left side of the vertical barrier on the left is \textit{identified} with the right side of the vertical barrier on the right.
Thus a flow line that hits the left side of the vertical barrier on the left continues in the same direction from the corresponding
point on the right side of the vertical barrier on the right.
On the other hand, the left side of the vertical barrier on the right is also \textit{identified} with the right side of the vertical barrier on the right.
Thus a flow line that hits the left side of the vertical barrier on the right continues straight through in the same direction.
Indeed, the $1$-direction flows indicated in Figure~3 correspond precisely to the billiard flows in Figure~1.

\begin{displaymath}
\begin{array}{c}
\includegraphics[scale=0.8]{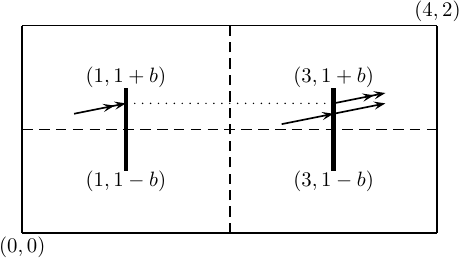}
\\
\mbox{Figure 3: geodesic flow on the system in Figure~2}
\end{array}
\end{displaymath}

Note that the system in Figure~3 is modified from a surface, but is not a surface.
The left side of both the vertical barrier on the left and the vertical barrier on the right are \textit{identified} with the right side
of the vertical barrier on the right, so there is no pairwise identification of edges.
The important point is that we can still study $1$-direction flow on it,
although this flow is not time reversible.

The systems in Figures 3 and~4 are equivalent.
The latter dynamical system satisfies the following convention which we now adopt.

\begin{displaymath}
\begin{array}{c}
\includegraphics[scale=0.8]{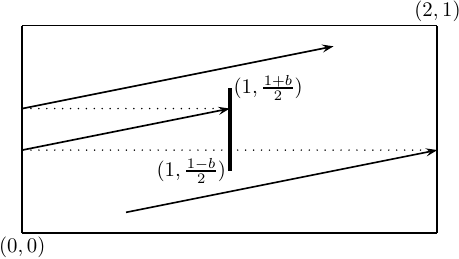}
\\
\mbox{Figure 4: a system equivalent to Figures 2 and~3}
\end{array}
\end{displaymath}

\begin{convention}
Henceforth we consider \textit{dissipative systems} involving $1$-direction
flow from left to right and with irrational slope $\alpha>0$.
When the \textit{dissipative flow line} hits a one-sided barrier, it then continues from the corresponding point on the left vertical edge
of the system and in the same direction.
\end{convention}

Consider the dissipative system $\PPP$ obtained from the $2$-square torus modified by the inclusion of a one-sided barrier
which may include both, one or neither of the endpoints $(1,a)$ and $(1,b)$, as shown in Figure~5.
We see that the two distinct flow line segments below the dotted line both continue along the bold
flow line segment.
Thus any measure preserving property of the flow is violated by a factor of~$2$, and so we cannot directly apply ergodic theory.

\begin{displaymath}
\begin{array}{c}
\includegraphics[scale=0.8]{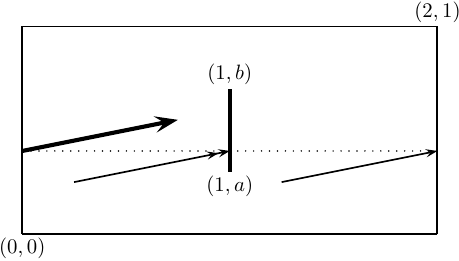}
\\
\mbox{Figure 5: lack of measure preserving property}
\end{array}
\end{displaymath}

For the system $\PPP$ shown in Figure~5, the dissipative flow of fixed irrational slope $\alpha>0$
can be described in terms of an operator $\T_\alpha^{(t)}$, where $t\ge0$ denotes time.
Roughly speaking, this operator denotes advancing geodesic flow of the fixed slope $\alpha$ by time~$t$.
More precisely, for any dissipative flow line $\LLL(t)$, $t\ge0$, of the fixed slope~$\alpha$, we write
\textcolor{white}{xxxxxxxxxxxxxxxxxxxxxxxxxxxxxx}
\begin{displaymath}
\T_\alpha^{(t)}(\LLL(t_1))=\LLL(t_1+t),
\quad
t_1\ge0.
\end{displaymath}
In other words, the operator $\T_\alpha^{(t)}$ advances the point $\LLL(t_1)$ on the flow line to the point $\LLL(t_1+t)$
after a further time $t$ has elapsed.

Using this operator $\T_\alpha^{(t)}$, we can partition $\PPP$ into two subsets.
To define these two subsets, we use some concepts introduced by Birkhoff~\cite{birkhoff27}
for a very general class of dynamical systems in 1927.

We say that a point $P\in\PPP$ is a \textit{wandering point} of $\PPP$ if there exists an open neighbourhood $D=D(P)$ of the point $P$
and a finite threshold $t_0\ge0$ such that the intersection
\textcolor{white}{xxxxxxxxxxxxxxxxxxxxxxxxxxxxxx}
\begin{equation}\label{eq2.1}
D\cap \T_\alpha^{(t)}(D)=\emptyset
\quad
\mbox{for every $t>t_0$};
\end{equation}
so that a dissipative flow line of slope $\alpha$ and starting from $P$ is bounded away from $P$ after time~$t_0$.
Furthermore, we say that a point $P\in\PPP$ is a \textit{non-wandering point} of $\PPP$ if
it is not a wandering point of~$\PPP$.

The set $\WWW(\PPP;\alpha)$ of all wandering points of $\PPP$ is called the \textit{transient set} of~$\PPP$,
and the set $\RRR(\PPP;\alpha)$ of all non-wandering points of $\PPP$ is called the \textit{recurrent set} or \textit{attractor} of~$\PPP$.
These two subsets of $\PPP$ are complements of each other, and give a partition of~$\PPP$.

It follows easily from the definition that the transient set $\WWW(\PPP;\alpha)$ is open
and that the recurrent set $\RRR(\PPP;\alpha)$ is closed and invariant under dissipative flow of slope~$\alpha$.
These two subsets are illustrated in Figure~6 in the case of the modified $2$-square torus in Figure~5.
We emphasize in particular their dependence on the fixed slope $\alpha$ of the flow under consideration.
It is also worth noting that the definition of wandering points and non-wandering points is given in terms of open neighbourhoods
of the points and their images under the operator $\T_\alpha^{(t)}$, and not in terms of the points and dissipative flow lines
that start from them.

\begin{displaymath}
\begin{array}{c}
\includegraphics[scale=0.8]{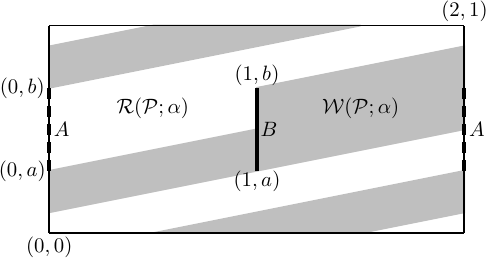}
\\
\mbox{Figure 6: transient set and recurrent set}
\end{array}
\end{displaymath}

Figure~7 illustrates that any dissipative flow line of slope $\alpha$ that starts from a point $P\in\RRR(\PPP;\alpha)$
remains in the recurrent set $\RRR(\PPP;\alpha)$, so that
the recurrent set $\RRR(\PPP;\alpha)$ is invariant under dissipative flow of slope~$\alpha$.

\begin{displaymath}
\begin{array}{c}
\includegraphics[scale=0.8]{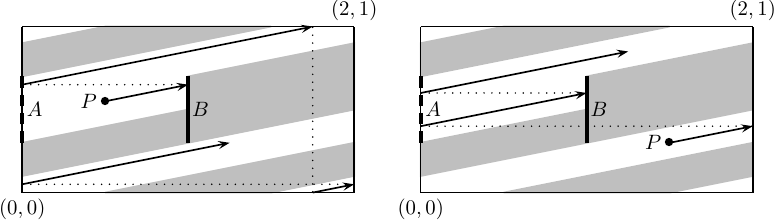}
\\
\mbox{Figure 7: invariance of the recurrent set under dissipative flow}
\end{array}
\end{displaymath}

Meanwhile, Figure~8 illustrates that any dissipative flow line of slope $\alpha$ that starts from a point $P\in\WWW(\PPP;\alpha)$
eventually moves into the recurrent set $\RRR(\PPP;\alpha)$.
Once there, it clearly never returns to the transient set $\WWW(\PPP;\alpha)$.

\begin{displaymath}
\begin{array}{c}
\includegraphics[scale=0.8]{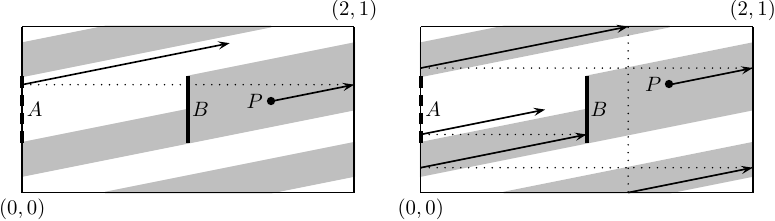}
\\
\mbox{Figure 8: leaving the transient set to join the recurrent set}
\end{array}
\end{displaymath}

We comment here that the above definitions for $\WWW(\PPP;\alpha)$ and $\RRR(\PPP;\alpha)$
remain valid for any system $\PPP$ that arises when a finite polysquare translation surface
is modified by the addition of a one-sided barrier on the common vertical edge of $2$ atomic squares.
However, we need to pay some attention to the existence of \textit{pathological points}.
These are points $P\in\PPP$ where a dissipative flow line of slope $\alpha$ that starts from the point $P$
hits a singularity of $\PPP$ and becomes undefined.
For the modified $2$-square torus in Figures 5 and~6, provided that we clearly
indicate whether the one-sided barrier includes either endpoint, then while these endpoints are singularities,
they are not pathological points because we have clear indication on how the flow line continues
after hitting either of them.
Pathological points do not affect our argument, as they form a set of $2$-dimensional Lebesgue measure~$0$.

\begin{theorem}\label{thm1}
Consider dissipative flow of slope~$\alpha$, where $\alpha>0$ is irrational, on a system $\PPP$ such as those
shown in Figures 4--8,
where the $2$-square torus has been modified with the inclusion of a one-sided barrier on the common vertical edge
of the atomic squares, in the form of a union of finitely many vertical intervals.
For the recurrent set $\RRR(\PPP;\alpha)$ and the transient set $\WWW(\PPP;\alpha)$, the following hold:

\emph{(i)}
Both sets are finite unions of polygons with total area $1$ and with boundary edges that are vertical or of slope~$\alpha$.

\emph{(ii)}
The recurrent set $\RRR(\PPP;\alpha)$ is invariant under dissipative flow of slope~$\alpha$, and
every half-infinite dissipative flow line of slope $\alpha$ in the set $\RRR(\PPP;\alpha)$ is uniformly distributed in this set.

\emph{(iii)}
Every half-infinite dissipative flow line of slope $\alpha$ that starts from a point in the transient set $\WWW(\PPP;\alpha)$
eventually leaves the set, moves to the set $\RRR(\PPP;\alpha)$ and is uniformly distributed there.
\end{theorem}

\begin{proof}
We proceed in a number of steps.

\begin{step1}
The pointwise definition of the recurrent set $\RRR(\PPP;\alpha)$ does not provide any hint on how we may explicitly
construct and describe the set.
Indeed, in the general situation, this is a rather difficult problem.
However, for the modified $2$-square torus under consideration, we consider a no-go zone $\NNN(\PPP;\alpha)$
illustrated in Figure~9.
This is an open subset of $\PPP$ obtained by sweeping the barrier $B$ without its endpoints by the dissipative flow of slope $\alpha$
towards the right vertical edge of~$\PPP$.

\begin{displaymath}
\begin{array}{c}
\includegraphics[scale=0.8]{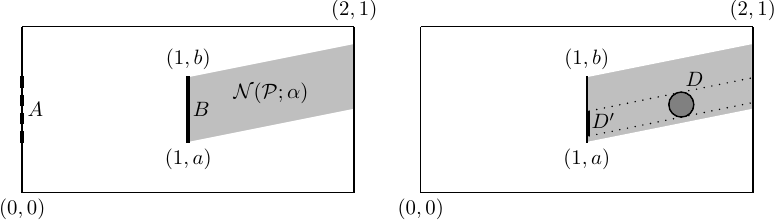}
\\
\mbox{Figure 9: the $2$-square torus modified with one-sided barriers}
\end{array}
\end{displaymath}

We claim that $\NNN(\PPP;\alpha)\subset\WWW(\PPP;\alpha)$;
in other words, that every point of the set $\NNN(\PPP;\alpha)$ is a wandering point of~$\PPP$.
Since $\NNN(\PPP;\alpha)$ is open, every point of $\NNN(\PPP;\alpha)$ is contained in an open subset $D\subset\NNN(\PPP;\alpha)$.
To justify this claim, it suffices to show that for every open subset $D\subset\NNN(\PPP;\alpha)$,
there exists $t_0>0$ such that \eqref{eq2.1} holds.
Now let $D'$ denote the image of $D$ on the right side of the one-sided barrier $B$ under projection in the direction $(-1,-\alpha)$,
as illustrated in the picture on the right in Figure~9.
Then $D'$ is an open subinterval of $B$ and
\begin{displaymath}
D'\cap \T_\alpha^{(t)}(D)=\emptyset
\quad\mbox{for every $t>0$},
\end{displaymath}
so that for every $t_0$ satisfying $0\le t_0\le(1+\alpha^2)^{1/2}$,
\begin{displaymath}
\T_\alpha^{(t_0)}(D')\cap \T_\alpha^{(t)}(D)=\emptyset
\quad\mbox{for every $t>(1+\alpha^2)^{1/2}$}.
\end{displaymath}
It is not difficult to see that
\textcolor{white}{xxxxxxxxxxxxxxxxxxxxxxxxxxxxxx}
\begin{displaymath}
D\subset\bigcup_{0\le t_0\le(1+\alpha^2)^{1/2}}\T_\alpha^{(t_0)}(D').
\end{displaymath}
It then follows that
\textcolor{white}{xxxxxxxxxxxxxxxxxxxxxxxxxxxxxx}
\begin{equation}\label{eq2.2}
D\cap \T_\alpha^{(t)}(D)=\emptyset
\quad\mbox{for every $t>(1+\alpha^2)^{1/2}$},
\end{equation}
justifying our claim.
\end{step1}

\begin{step2}
We define the open set $\frakM(\PPP;B;\alpha)$ to be the largest possible extension of the no-go zone $\NNN(\PPP;\alpha)$.
To find this set, let $A$ be an interval on the left vertical edge of~$\PPP$, identical to the interval $B$ that represents
the one-sided barrier, as shown in Figure~6.
Consistent with the fact that the transient set $\WWW(\PPP;\alpha)$ is open, we assume that the interval $B$ is open.
This assumption is valid, since any deviation from it affects only sets of $2$-dimensional Lebesgue measure~$0$.

\begin{rfr}
A given point $P\in\PPP$ satisfies $P\in\frakM(\PPP;B;\alpha)$ if and only if starting from~$P$,
an ordinary geodesic on the $2$-square torus in the direction $(-1,-\alpha)$ hits the open interval $B$
before hitting the corresponding interval $A$ on the left vertical edge of~$\PPP$.
\end{rfr}

Using time-quantitative equidistribution of geodesic flow on the $2$-square torus
in the form of the Koksma--Erd\H{o}s--Tur\'{a}n version of the Weyl criterion,
it follows that for any given irrational number $\alpha>0$, there is a finite threshold $T_0(\alpha)$ such that
starting from any point $P$ and moving in the direction $(-1,-\alpha)$, this geodesic segment hits $B$ or $A$
in time less than $T_0(\alpha)$.
Hence the set $\frakM(\PPP;B;\alpha)$ is a finite union of polygons with boundary edges
that are vertical or of slope~$\alpha$.
We shall comment further on the finiteness of $T_0(\alpha)$ in the Remark after the completion of the proof of Theorem~\ref{thm1}.

We claim that
\textcolor{white}{xxxxxxxxxxxxxxxxxxxxxxxxxxxxxx}
\begin{equation}\label{eq2.3}
\frakM(\PPP;B;\alpha)\subset\WWW(\PPP;\alpha);
\end{equation}
in other words, that every point of the set $\frakM(\PPP;B;\alpha)$ is a wandering point of~$\PPP$.
Since $\frakM(\PPP;B;\alpha)$ is open, every point of $\frakM(\PPP;B;\alpha)$ is contained in an open subset $D\subset\frakM(\PPP;B;\alpha)$.
To justify this claim, it suffices to show that for every open subset $D\subset\frakM(\PPP;B;\alpha)$,
there exists $t_0>0$ such that \eqref{eq2.1} holds.
An argument similar to that in Part~1 leads to an analogue of \eqref{eq2.2}, in the form
\begin{displaymath}
D\cap \T_\alpha^{(t)}(D)=\emptyset
\quad\mbox{for every $t>T_0(\alpha)$},
\end{displaymath}
justifying the assertion \eqref{eq2.3}.

Let $\frakM^c(\PPP;B;\alpha)=\PPP\setminus\frakM(\PPP;B;\alpha)$ denote the set of all points $P\in\PPP$
that are not in $\frakM(\PPP;B;\alpha)$.
Then
\textcolor{white}{xxxxxxxxxxxxxxxxxxxxxxxxxxxxxx}
\begin{equation}\label{eq2.4}
\RRR(\PPP;\alpha)\subset\frakM^c(\PPP;B;\alpha).
\end{equation}
The $2$-cyclic symmetry of the system $\PPP$ then ensures that the sets $\frakM(\PPP;B;\alpha)$ and $\frakM^c(\PPP;B;\alpha)$
have essentially the same projection on the unit torus $[0,1)^2$.
It follows that both sets have $2$-dimensional Lebesgue measure~$1$.
\end{step2}

\begin{step3}
Let $\LLL(t)$, $t\ge0$, denote any half-infinite dissipative flow line on~$\PPP$.
We say that a point $P\in\PPP$ is an \textit{infinite-time limit point} of $\LLL$ if there exists a sequence of
time instances $t_1<t_2<t_3<\ldots$ such that $t_i\to\infty$ and $\LLL(t_i)\to P$ as $i\to\infty$.
Furthermore, we say that the collection of all infinite-time limit points of $\LLL$ is
the infinite-time limit set of~$\LLL$.
By definition, this set is invariant under dissipative flow of slope $\alpha$ and so has $2$-dimensional Lebesgue measure
$0$, $1$ or~$2$.
This cannot be equal to~$0$, since its projection to the unit torus $[0,1)^2$ has $2$-dimensional Lebesgue measure~$1$.
This cannot be equal to~$2$, since it does not contain points of $\NNN(\PPP;\alpha)$.
Hence it has $2$-dimensional Lebesgue measure~$1$.

Suppose now that $\LLL(t)$, $t\ge0$, denotes any such half-infinite dissipative flow line on $\PPP$ that starts
from a point in the set $\frakM^c(\PPP;B;\alpha)$, so that $\LLL(0)\in\frakM^c(\PPP;B;\alpha)$.
We claim that the entire dissipative flow line is disjoint from the set $\frakM(\PPP;B;\alpha)$,
so that the set $\frakM^c(\PPP;B;\alpha)$ is invariant under dissipative flow of slope~$\alpha$.
Suppose, on the contrary, that there exists $t_1>0$ such that $\LLL(t_1)\in\frakM(\PPP;B;\alpha)$.
Then it follows from the reverse flow recipe that there exists $t_0<t_1$ such that $\LLL(t_0)\in B$.
There are two possibilities:
(i)
If $t_0>0$, then the dissipative flow of slope $\alpha$ cannot take $\LLL(0)$ to $\LLL(t_0)$, a contradiction.
(ii)
If $t_0\le0$, then it follows that $\LLL(0)\in\frakM(\PPP;B;\alpha)$, also a contradiction.
Thus the infinite-time limit set of $\LLL$ is a subset of $\frakM^c(\PPP;B;\alpha)$.
Since both sets have $2$-dimensional Lebesgue measure~$1$, they are essentially equal.
\end{step3}

\begin{step4}
The essentially $2$-cyclic symmetry of the system $\PPP$ guarantees that the sets $\frakM(\PPP;B;\alpha)$ and $\frakM^c(\PPP;B;\alpha)$
split up the system $\PPP$ in an essentially fair way.
More precisely, for every $(x,y)\in[0,1)^2$, apart from a set of $2$-dimensional Lebesgue measure~$0$, precisely one of the two points
\begin{displaymath}
(x,y)
\quad\mbox{and}\quad
(x+1,y)
\end{displaymath}
lies in $\frakM(\PPP;B;\alpha)$ while the other lies in $\frakM^c(\PPP;B;\alpha)$.
We next use Figure~6 to guide our discussion.

Starting from anywhere within the set $\frakM^c(\PPP;B;\alpha)$,
the dissipative flow of slope $\alpha$ moves within the set $\frakM^c(\PPP;B;\alpha)$
which is the white part in Figure~6.
When it reaches the white part immediately to the left of $B$ or the white part immediately to the left of $A$ on the right vertical edge,
it then \textit{jumps} back to~$A$.
The fair split of $\PPP$ mentioned above implies that this jump is $1$-to-$1$ and measure preserving.
Thus the dissipative flow of slope $\alpha$ restricted to the set $\frakM^c(\PPP;B;\alpha)$ is both invertible and measure preserving.

Suppose that we start from $A$ and sweep it along by the dissipative flow of slope~$\alpha$.
Then this certainly contains in the left atomic square the twin of $\NNN(\PPP;\alpha)$ in the right atomic square,
as illustrated in Figure~10.

\begin{displaymath}
\begin{array}{c}
\includegraphics[scale=0.8]{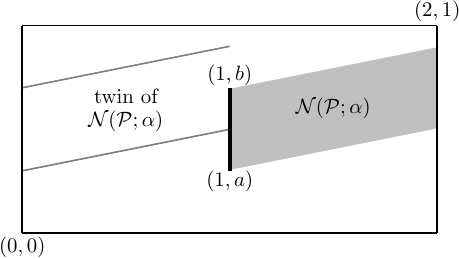}
\\
\mbox{Figure 10: no-go zone and its twin}
\end{array}
\end{displaymath}

Furthermore, this flow returns to $A$ infinitely many times in a $1$-to-$1$ and measure preserving way.
Let $\frakR(\PPP;A;\alpha)$ be the resulting subset of~$\PPP$,
the construction of which requires only finitely many zig-zaggings; see the Remark after the completion of the proof of Theorem~\ref{thm1}.
It is also clear that
\begin{equation}\label{eq2.5}
\frakR(\PPP;A;\alpha)\subset\frakM^c(\PPP;B;\alpha).
\end{equation}
Observe next that $\PPP$ is a square-covering system, so the projection of the dissipative flow
of slope $\alpha$ to the unit torus $[0,1)^2$ is ordinary integrable geodesic flow of slope~$\alpha$.
Since $\alpha$ is irrational, the dissipative flow of slope $\alpha$ modulo one is \textit{ergodic}.
It then follows that any subset of $\PPP$ of positive $2$-dimensional Lebesgue measure that is invariant under
the dissipative flow of slope $\alpha$ must have $2$-dimensional Lebesgue measure $1$ or~$2$.
The inclusion \eqref{eq2.5} then clearly implies that $\frakR(\PPP;A;\alpha)$ must have $2$-dimensional Lebesgue measure~$1$.
This means that
\begin{equation}\label{eq2.6}
\frakR(\PPP;A;\alpha)=\frakM^c(\PPP;B;\alpha),
\end{equation}
apart possibly from an exceptional set of $2$-dimensional Lebesgue measure~$0$.

The measure preserving aspect then allows us to use the well known Poincare recurrence theorem,
and conclude that apart possibly from an exceptional set of $2$-dimensional Lebesgue measure~$0$,
every point in $\frakR(\PPP;A;\alpha)$ is a non-wandering point and so lies in $\RRR(\PPP;\alpha)$.
Combining this observation with \eqref{eq2.4} and \eqref{eq2.6}, we conclude that
\textcolor{white}{xxxxxxxxxxxxxxxxxxxxxxxxxxxxxx}
\begin{displaymath}
\frakR(\PPP;A;\alpha)=\frakM^c(\PPP;B;\alpha)=\RRR(\PPP;\alpha),
\end{displaymath}
apart possibly from exceptional sets of $2$-dimensional Lebesgue measure~$0$.

On the other hand, the dissipative flow of slope $\alpha$ moves $B$ first to the gray part in Figure~6
which is $\NNN(\PPP;\alpha)$.
When it reaches the gray part immediately to the left of $B$ or the gray part
immediately to the left of $A$ on right vertical edge, it then jumps to~$A$.
Thus every point of $\frakM(\PPP;B;\alpha)$ is eventually moved to $A$ by the flow.
Since both sets $\WWW(\PPP;\alpha)$ and $\frakM(\PPP;B;\alpha)$ have $2$-dimensional Lebesgue measure~$1$,
it follows that
\textcolor{white}{xxxxxxxxxxxxxxxxxxxxxxxxxxxxxx}
\begin{displaymath}
\WWW(\PPP;\alpha)=\frakM(\PPP;B;\alpha),
\end{displaymath}
apart possibly from an exceptional set of $2$-dimensional Lebesgue measure~$0$.
\end{step4}

\begin{step5}
Uniformity in the recurrent set $\RRR(\PPP;\alpha)$ is now established by projecting to the unit torus $[0,1)^2$
and applying the results such as the Birkhoff ergodic theorem and also unique ergodicity.
\end{step5}

This completes the proof of Theorem~\ref{thm1}.
\end{proof}

\begin{remark}
The finiteness of $T_0(\alpha)$ in Step~2 of the preceding proof follows from the fact that
we are essentially considering geodesic flow in the unit torus $[0,1)^2$ in the direction $(-1,-\alpha)$,
and $\NNN(\PPP;\alpha)$ is extended to $\frakM(\PPP;B;\alpha)$
in a finite number of zig-zaggings; \textit{i.e.} a finite number of traverses between the opposite vertical edges
of atomic squares.
We also use the following fact from diophantine approximation.

\begin{irl}
Consider the irrational rotation sequence
\begin{displaymath}
s_0+n\alpha,
\quad
n=0,1,2,3,\ldots,
\end{displaymath}
in the unit torus $[0,1)$, where $\alpha$ is irrational.
For any fixed $\eps>0$, there is a finite threshold $N=N(\alpha;\eps)$ such that the union
\begin{displaymath}
\bigcup_{n=0}^{N-1}(s_0+n\alpha-\eps,s_0+n\alpha+\eps)=[0,1)
\end{displaymath}
for any starting point~$s_0$.
\end{irl}

\begin{proof}
A fundamental result in diophantine approximation says that
\begin{displaymath}
\left\vert\alpha-\frac{p_k}{q_k}\right\vert<\frac{1}{q_kq_{k+1}},
\end{displaymath}
where $p_k/q_k$, $k=0,1,2,3,\ldots,$ is the sequence of convergents to~$\alpha$.
Then
\begin{equation}\label{eq2.7}
\left\vert\ell\alpha-\frac{\ell p_k}{q_k}\right\vert<\frac{\ell}{q_kq_{k+1}}\le\frac{1}{q_{k+1}}<\frac{1}{q_k},
\quad
\ell=0,1,2,3,\ldots,q_k.
\end{equation}
Furthermore, $p_k$ and $q_k$ are coprime, so that as $\ell$ runs through the numbers $1,\ldots,q_k$,
$\ell p_k$ runs through a complete set of residues modulo~$q_k$.
Choosing a positive integer $k$ such that the denominator $q_k>1/\eps$,
we see that the threshold $N=N(\alpha;\eps)=q_k$ satisfies the requirements.
\end{proof}
\end{remark}

%
%

\section{Extension of the modified Veech model}\label{sec3}

Theorem~\ref{thm1} can be extended to the case of the $n$-square torus modified by identical one-sided
barriers on the vertical edges of adjoining atomic squares, in the form of a union of finitely many vertical intervals.
Denote this system by~$\PPP$.
Here Figure~11 illustrates $\PPP$ in the special case $n=4$ and where the union consists of a single vertical interval.

\begin{displaymath}
\begin{array}{c}
\includegraphics[scale=0.8]{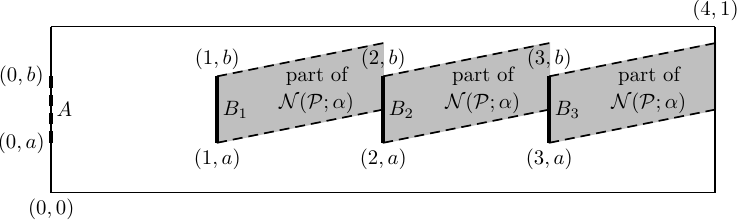}
\\
\mbox{Figure 11: the $4$-square torus modified with one-sided barriers}
\end{array}
\end{displaymath}

Consider a dissipative flow line on $\PPP$ with irrational slope $\alpha>0$.
When this flow line hits the left side of any one-sided barrier $B_1,\ldots,B_{n-1}$, it then continues in the same direction
from the corresponding point on the vertical interval $A$ on the left vertical edge.

We can prove that the recurrent set $\RRR(\PPP;\alpha)$ has area $1$ and is a finite union of polygons
with boundary edges that are vertical or of slope~$\alpha$,
and also that the transient set $\WWW(\PPP;\alpha)$ consists of $n-1$ parts that are,
apart from exceptional sets of $2$-dimensional Lebesgue measure~$0$,
all equivalent to $\RRR(\PPP;\alpha)$ modulo one.
Furthermore, every dissipative flow line with slope $\alpha$ in the recurrent set $\RRR(\PPP;\alpha)$ is uniformly
distributed in this set.
Moreover, every dissipative flow line with slope $\alpha$ that starts from any point in the transient set $\WWW(\PPP;\alpha)$
eventually leaves the set, moves to the recurrent set $\RRR(\PPP;\alpha)$ and is uniformly distributed there.

While we have considered one-sided barriers, there is still some sort of cyclic symmetry, of great help so far.
However, the situation becomes considerably more complicated if this cyclic symmetry is also violated, as in the system illustrated
in Figure~12 concerning a modified $n$-square torus in the case $n=4$.

\begin{displaymath}
\begin{array}{c}
\includegraphics[scale=0.8]{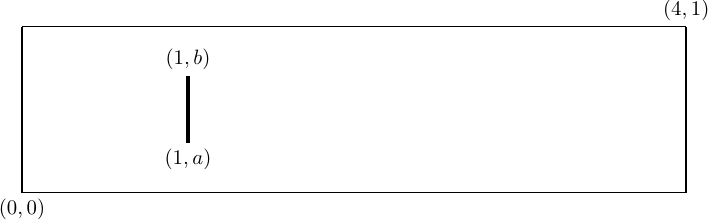}
\\
\mbox{Figure 12: a dissipative system while cyclic symmetry is violated}
\end{array}
\end{displaymath}

In this section, we establish the following simpler result.

\begin{theorem}[weak area theorem]\label{thmIDS.2}
Given any $\eps>0$, there exists an irrational number $\alpha>0$ and a positive integer $n$ such that
for some $n$-square torus which is modified with the addition of a one-sided barrier,
the time evolution of the dissipative flow of slope $\alpha$ contracts the initial space-volume
to less than $\eps$-part.
In other words, denoting the modified $n$-square torus by~$\PPP$, the recurrent set $\RRR(\PPP;\alpha)$ satisfies
\begin{displaymath}
\frac{\lambda_2(\RRR(\PPP;\alpha))}{n}<\eps.
\end{displaymath}
\end{theorem}

As usual, $\lambda_2$ denotes the $2$-dimensional Lebesgue measure.

In fact, we establish the following result from which Theorem~\ref{thmIDS.2} follows easily from the details of the construction
by the application of a linear transformation.

\begin{theorem}[small attractor theorem]\label{thmIDS.3}
Suppose that for an irrational number $\alpha>0$, the continued fraction digits are unbounded.
Then there exists an infinite sequence of vertical intervals $B_\sigma$, $\sigma=1,2,3,\ldots,$ with the following property.
For every $\sigma=1,2,3,\ldots,$ let $\PPP_\sigma$ denote the unit torus modified by the addition of a vertical one-sided barrier $B_\sigma$
as illustrated in Figure~13,
and consider the recurrent set $\RRR(\PPP_\sigma;\alpha)$ under dissipative flow of slope $\alpha$ on~$\PPP_\sigma$.
Then
\begin{displaymath}
\lambda_2(\RRR(\PPP_\sigma;\alpha))\to0
\quad
\mbox{as $\sigma\to\infty$}.
\end{displaymath}
\end{theorem}

\begin{displaymath}
\begin{array}{c}
\includegraphics[scale=0.8]{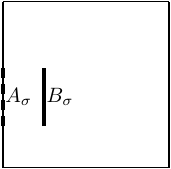}
\\
\mbox{Figure 13: a one-sided barrier $B_\sigma$ and its counterpart $A_\sigma$ on the left vertical edge}
\end{array}
\end{displaymath}

\begin{proof}
For a one-sided vertical barrier $B_\sigma$ in the torus $[0,1)^2$,
let $A_\sigma=A_\sigma(B_\sigma)$ denote its counterpart on the left vertical edge of the torus,
as shown in Figure~13.

We partition the unit torus into two polygonal regions $\frakM(A_\sigma)$ and $\frakM(B_\sigma)$ using the reverse flow recipe
involving the flow in the direction $(-1,-\alpha)$.
This flow moves every point of $\frakM(A_\sigma)$ to $A_\sigma$ without hitting~$B_\sigma$,
and moves every point of $\frakM(B_\sigma)$ to $B_\sigma$ without hitting~$A_\sigma$.
It is easy to see that $\frakM(A_\sigma)$ is invariant under dissipative flow of slope~$\alpha$,
and that the dissipative flow of slope $\alpha$ moves $\frakM(B_\sigma)$ to $\frakM(A_\sigma)$.
Here the last point follows from the observation that the barrier-free geodesic flow of slope $\alpha$
moves $\frakM(B_\sigma)$ to $A_\sigma\cup B_\sigma$, and the dissipative flow of slope $\alpha$ jumps from $B_\sigma$ to~$A_\sigma$,
after which it will never return to $\frakM(B_\sigma)$.
Hence
\begin{displaymath}
\RRR(\PPP_\sigma;\alpha)\subset\frakM(A_\sigma),
\end{displaymath}
if we ignore a possible exceptional set of $2$-dimensional Lebesgue measure~$0$.
Thus to establish the small attractor theorem, it remains to show that $\frakM(A_\sigma)$ is less than $\eps$-part
of the unit torus $[0,1)^2$ provided that the index $\sigma$ is chosen appropriately.

Let $n\ge9$ be an integer.
For every integer $i=1,2,3,\ldots,n$, let
\begin{displaymath}
V_i=\left\{\left(\frac{i}{n},y\right):y\in[0,1)\right\}
\end{displaymath}
denote the vertical line segment of length $1$ and $x$-coordinate equal to $i/n$ in the unit torus $[0,1)^2$.
We consider a geodesic $\LLL_\alpha$ on $[0,1)^2$, with irrational slope $\alpha>0$ and starting from the point $(0,\beta)$
on the left vertical edge of $[0,1)^2$.
Then $\LLL_\alpha$ intersects $V_i$ for the first time at the point $(i/n,\{\beta+i\alpha/n\})$,
and the collection of the intersection points of $\LLL_\alpha$ with $V_i$ is given by
\begin{displaymath}
\left(\frac{i}{n},\left\{\beta+\frac{(i+jn)\alpha}{n}\right\}\right),
\quad
j=0,1,2,3,\ldots.
\end{displaymath}
These points partition $V_i$ in a fairly regular fashion.
More precisely, let $q_k=q_k(\alpha)$ denote the denominator of the $k$-th convergent to the number~$\alpha$,
where the integer $k$ is assumed to be sufficiently large.
Consider the finite collection
\begin{equation}\label{eq3.1}
\left(\frac{i}{n},\left\{\beta+\frac{(i+jn)\alpha}{n}\right\}\right),
\quad
j=0,1,2,3,\ldots,q_k-1,
\end{equation}
of the first $q_k$ intersection points of $\LLL_\alpha$ with~$V_i$.
Using the inequality \eqref{eq2.7} and the inequality $\Vert\gamma_1-\gamma_2\Vert\le\vert\{\gamma_1\}-\{\gamma_2\}\vert$,
where $\Vert\gamma\Vert$ denotes the distance of $\gamma$ from the nearest integer,
we see that any two distinct points in \eqref{eq3.1} represented by $j_1$ and $j_2$ satisfying $0\le j_1<j_2\le q_k-1$ have gap
\begin{align}\label{eq3.2}
&
\left\vert\left\{\beta+\frac{(i+j_1n)\alpha}{n}\right\}-\left\{\beta+\frac{(i+j_2n)\alpha}{n}\right\}\right\vert
\nonumber
\\
&\quad
\ge\Vert(j_1-j_2)\alpha\Vert
\ge\left\Vert\frac{(j_1-j_2)p_k}{q_k}\right\Vert-\frac{1}{q_{k+1}}
\ge\frac{1}{q_k}-\frac{1}{q_{k+1}},
\end{align}
since $\vert j_1-j_2\vert<q_k$ and the pair $p_k,q_k$ are coprime, so that $(j_1-j_2)p_k$ is not divisible by~$q_k$.
Thus the $y$-coordinates of the points in the set \eqref{eq3.1} closely mimic an arithmetic progression with gap $1/q_k$.

Consider next a subinterval $I$ on the left vertical edge of $[0,1)^2$, with lower and upper endpoints
given respectively by
\begin{displaymath}
(0,\beta)
\quad\mbox{and}\quad
\left(0,\beta+\frac{1}{q_k}-\frac{1}{q_{k+1}}\right).
\end{displaymath}
We move $I$ by using the splitting-free geodesic flow of slope $\alpha$ on $[0,1)^2$ so that the horizontal distance
covered is~$q_k$, and let $S(I)$ denote the resulting strip in $[0,1)^2$.
In view of the inequality \eqref{eq3.2}, we see that this strip $S(I)$ is self-avoiding, and its area is precisely
\textcolor{white}{xxxxxxxxxxxxxxxxxxxxxxxxxxxxxx}
\begin{displaymath}
q_k\left(\frac{1}{q_k}-\frac{1}{q_{k+1}}\right)=1-\frac{q_k}{q_{k+1}}.
\end{displaymath}
Furthermore, let
\textcolor{white}{xxxxxxxxxxxxxxxxxxxxxxxxxxxxxx}
\begin{equation}\label{eq3.3}
I_\nu,
\quad
\nu=1,2,3,\ldots,nq_k,
\end{equation}
denote the successive vertical intervals of intersection of $S(I)$ with the line segments $V_i$, $i=1,2,3,\ldots,n$,
so that if $\nu=i+jn$, where the indices $i=1,2,3,\ldots,n$ and $j=0,1,2,3,\ldots,q_k-1$, then the interval $I_\nu$ has lower endpoint
given by \eqref{eq3.1}.

Consider next a subcollection
\begin{equation}\label{eq3.4}
I_\nu,
\quad
\nu=1,2,3,\ldots,[n^{1/2}q_k],
\end{equation}
of the collection of intervals \eqref{eq3.3}, obtained in the early stage of the process in the construction of the set $S(I)$.
This gives \textit{time-closeness} but not \textit{space-closeness}.
The total length of these intervals is at least
\begin{displaymath}
(n^{1/2}-1)q_k\left(\frac{1}{q_k}-\frac{1}{q_{k+1}}\right)>n^{1/2}-2,
\end{displaymath}
provided that
\textcolor{white}{xxxxxxxxxxxxxxxxxxxxxxxxxxxxxx}
\begin{equation}\label{eq3.5}
q_{k+1}>n^{1/2}q_k.
\end{equation}
Hence there is some $y_0\in(0,1)$ such that at least $n^{1/2}-2$ intervals in \eqref{eq3.4}
contain points with $y$-coordinates equal to~$y_0$.
The $y$-coordinates of the lower endpoints of these intervals are clearly contained in an interval of length less than $1/q_k$,
so it follows immediately from the pigeonhole principle that the $y$-coordinates of two of these intervals $I_{\nu_1}$ and $I_{\nu_2}$
differ by at most $1/(n^{1/2}-2)q_k$.
Thus there are subintervals $I'_{\nu_1}\subset I_{\nu_1}$ and $I'_{\nu_2}\subset I_{\nu_2}$ that contain precisely the same $y$-coordinates,
with common length satisfying
\begin{displaymath}
\vert I'_{\nu_1}\vert=\vert I'_{\nu_2}\vert
\ge\vert I\vert-\frac{1}{(n^{1/2}-2)q_k}
=\frac{1}{q_k}-\frac{1}{q_{k+1}}-\frac{1}{(n^{1/2}-2)q_k}
>\frac{1}{q_k}-\frac{4}{n^{1/2}q_k},
\end{displaymath}
provided that \eqref{eq3.5} holds and $n\ge9$.

We may assume, without loss of generality, that $\nu_1<\nu_2$.
Clearly the reverse $\alpha$-flow, starting from the far end of $S(I)$, namely $I_{nq_k}$, hits $I_{\nu_2}$ before hitting~$I_{\nu_1}$.
We now let $B=I'_{\nu_2}$ and $A=I'_{\nu_1}$, and let $\PPP$ denote the unit torus $[0,1)^2$ modified by the inclusion of
the one-sided barrier~$B$.
Then clearly the transient set has measure
\begin{displaymath}
\lambda_2(\WWW(\PPP;\alpha))
\ge\frac{1}{n}(nq_k-n^{1/2}q_k)\left(\frac{1}{q_k}-\frac{4}{n^{1/2}q_k}\right)>1-\frac{4}{n^{1/2}},
\end{displaymath}
so that the recurrent set has measure
\begin{displaymath}
\lambda_2(\RRR(\PPP;\alpha))<\frac{4}{n^{1/2}}<\eps,
\end{displaymath}
provided that $n\ge(4/\eps)^2$.
Naturally, since $q_{k+1}>a_{k+1}q_k$ and the continued fraction digits $a_k$ of $\alpha$ are unbounded,
we can always find a sufficiently large integer $k$ so that \eqref{eq3.5} is satisfied.

Finally we can determine that the vertical line segment $V_i$, $i=1,2,3,\ldots,n$,
that contains the edge $A$ to be the left vertical edge of the unit torus, and this completes the proof.
\end{proof}

%
%

\section{Modification of polysquare translation surfaces}\label{sec4}

We next consider modification of a polysquare translation surface
by the inclusion of a one-sided vertical barrier on part of a vertical edge.
We look at some examples.

\begin{example}\label{ex21}
Consider the horizontally reversed L-surface that is modified by the inclusion of a one-sided barrier $B$ of length~$b$, where $0<b<1$,
as shown in the picture on the left in Figure~14.
Here we make the assumption that the slope $\alpha>0$ of the dissipative flow satisfies $\alpha<b$.
Using the reverse flow recipe, we see that the transient set is essentially the part of the system that is coloured gray,
and includes the entire top atomic square.
Furthermore, if we modify the original reverse L-surface by including extra atomic squares on top of the top atomic square,
as illustrated in the picture on the right in Figure~14, we see that these extra atomic squares are all
part of the transient set.
This illustrates the possibility that the transient set may make up the overwhelming majority of the system.

\begin{displaymath}
\begin{array}{c}
\includegraphics[scale=0.8]{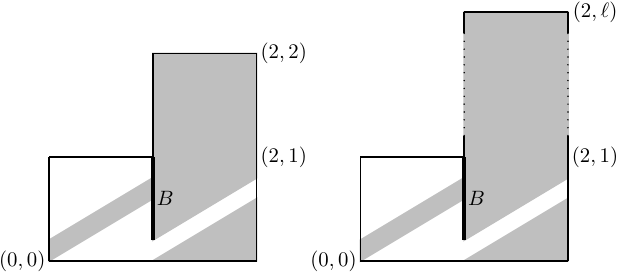}
\\
\mbox{Figure 14: horizontally reversed L-surface modified by the inclusion}
\\
\mbox{of a one-sided barrier}
\end{array}
\end{displaymath}
\end{example}

\begin{example}\label{ex22}
Consider the vertically reversed L-surface that is modified by the inclusion of a one-sided barrier $B$ of length~$b$, where $0<b<1$,
as shown in the picture on the left in Figure~15.
Again we make the assumption that the slope $\alpha>0$ of the dissipative flow satisfies $\alpha<b$.
Using the reverse flow recipe, we see that the transient set is essentially the part of the system that is coloured gray,
and leaves out the entire bottom atomic square.
Furthermore, if we modify the original reverse L-surface by including extra atomic squares below the bottom atomic square,
as illustrated in the picture on the right in Figure~15, we see that these extra atomic squares do not form
any part of the transient set.
This illustrates the possibility that the transient set may make up the overwhelming minority of the system.

\begin{displaymath}
\begin{array}{c}
\includegraphics[scale=0.8]{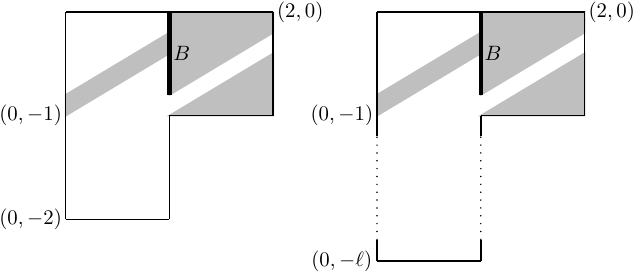}
\\
\mbox{Figure 15: vertically reversed L-surface modified by the inclusion}
\\
\mbox{of a one-sided barrier}
\end{array}
\end{displaymath}
\end{example}

\begin{example}\label{ex23}
Consider two separate copies of the modified horizontally reversed L-surface in Example~\ref{ex21},
with identically placed one-sided barriers $B_1$ and $B_2$ of length~$b$, where $0<b<1$, and glued
together one on top of the other as shown in Figure~16.
Again we make the assumption that the slope $\alpha>0$ of the dissipative flow satisfies $\alpha<b$.
Using the reverse flow recipe, we see that the transient set is essentially the part of the system that is coloured gray.
Clearly the recurrent set in white is made up of two non-trivial subsets, each of which is invariant
under dissipative flow of slope~$\alpha$.
Naturally the dissipative flow cannot take a point in one of the invariant subsets of the recurrent set to anywhere in the other invariant
subset, so it follows that the dissipative flow on the recurrent set cannot possibly be ergodic.

\begin{displaymath}
\begin{array}{c}
\includegraphics[scale=0.8]{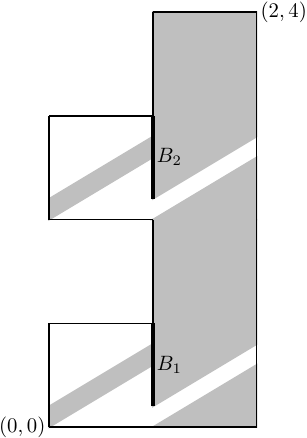}
\\
\mbox{Figure 16: L-surfaces glued together and modified by the inclusion}
\\
\mbox{of identical one-sided barriers}
\end{array}
\end{displaymath}
\end{example}

Consider the special case where we start with the L-surface and modify it with the inclusion of a one-sided barrier
on the vertical edge separating the two bottom atomic squares, as shown in Figure~17.

\begin{displaymath}
\begin{array}{c}
\includegraphics[scale=0.8]{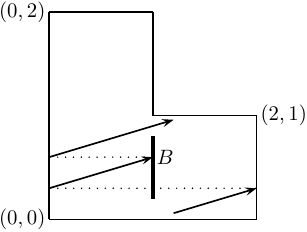}
\\
\mbox{Figure 17: L-surface modified by the inclusion of a one-sided barrier}
\end{array}
\end{displaymath}

Denote such a modified system by~$\PPP$.
Computer-generated pictures indicate that the recurrent set $\RRR(\PPP;\alpha)$ of $\PPP$ has area $1$ or~$2$,
and is a finite union of polygons with boundary edges that are vertical or of slope~$\alpha$.
The pictures can be surprisingly complicated, particularly if the barrier is short, or if it breaks into finitely many pieces.
This makes the problem rather interesting.

Meanwhile, the dissipative system $\PPP$ modulo one is the ordinary geodesic flow of slope $\alpha$
in the unit torus $[0,1)^2$, and this is uniquely ergodic.
This implies that if the recurrent set $\RRR(\PPP;\alpha)$ of $\PPP$ has area~$1$,
then the dissipative flow of slope $\alpha$ there is uniquely ergodic.
On the other hand, if the recurrent set $\RRR(\PPP;\alpha)$ of $\PPP$ has area~$2$,
then we do not know immediately whether the dissipative flow of slope $\alpha$ there is also uniquely ergodic.
In principle, the set $\RRR(\PPP;\alpha)$ may contain an invariant subset of area~$1$.
This raises an interesting uniformity question.

We next illustrate a class of dissipative, or \textit{non-conservative}, dynamical systems.
The restriction of the flow to the recurrent set gives rise to a \textit{nice conservative} measure preserving system,
to which we may apply classical ergodic theory.

Consider the system $\PPP$ shown in Figure~18,
where we start with the L-surface and then modify it by including a one-sided barrier of length $b$ as shown,
with lower endpoint $(1,0)$ and upper endpoint $(1,b)$.

\begin{displaymath}
\begin{array}{c}
\includegraphics[scale=0.8]{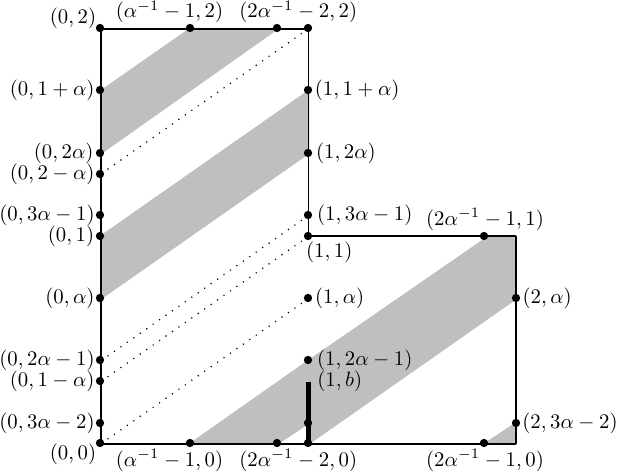}
\\
\mbox{Figure 18: L-surface modified by the inclusion of a one-sided barrier}
\end{array}
\end{displaymath}

We consider dissipative flow of slope $\alpha$ on~$\PPP$, and make the further restrictions
\begin{equation}\label{eq4.1}
\frac{2}{3}<\alpha<\frac{3}{4}
\quad\mbox{and}\quad
0<b<2\alpha-1,
\end{equation}
so that Figure~18 represents reasonable ranges of the parameters $b$ and~$\alpha$.

It is straightforward to check that the recurrent set $\RRR(\PPP;\alpha)$ is represented by the part of $\PPP$ in white,
and it has area~$2$.
The transient set $\WWW(\PPP;\alpha)$ is represented by the part of $\PPP$ in gray, and it has area~$1$.

We have highlighted in Figure~18 some critical points and line segments which
help us construct the relevant interval exchange transformation
of the dissipative flow of slope $\alpha$ on~$\PPP$.

Let $\EEE$ denote the union of the left vertical edges of the atomic squares of~$\PPP$.
It is not difficult to see that the intersection of $\EEE$ with the transient set $\WWW(\PPP;\alpha)$ modulo one
is the unit interval $[0,1)$, while the intersection
\begin{displaymath}
\VVV=\EEE\cap\RRR(\PPP;\alpha)
\end{displaymath}
of $\EEE$ with the recurrent set $\RRR(\PPP;\alpha)$ modulo one is also the unit interval $[0,1)$, but with multiplicity~$2$.

As shown in Figures 18 and~19, 
the set $\VVV$ is the union of four vertical line segments $L_1,L_2,L_3,L_4$, where

(1)
$L_1$ has endpoints $(0,0)$ and $(0,\alpha)$;

(2)
$L_2$ has endpoints $(0,1)$ and $(0,2\alpha)$;

(3)
$L_3$ has endpoints $(0,1+\alpha)$ and $(0,2)$; and

(4)
$L_4$ has endpoints $(1,2\alpha-1)$ and $(1,1)$.

We can make further subdivisions, as shown in Figures 18 and 19, where

(5)
$L_1$ is split into $L_{1,1},\ldots,L_{1,4}$ at the points $(0,3\alpha-2)$, $(0,1-\alpha)$ and $(0,2\alpha-1)$;

(6)
$L_2$ is split into $L_{2,1},\ldots,L_{2,3}$ at the points $(0,3\alpha-1)$ and  $(0,2-\alpha)$; and

(7)
$L_4$ is split into $L_{4,1},L_{4,2}$ at the point $(1,\alpha)$.

\begin{displaymath}
\begin{array}{c}
\includegraphics[scale=0.8]{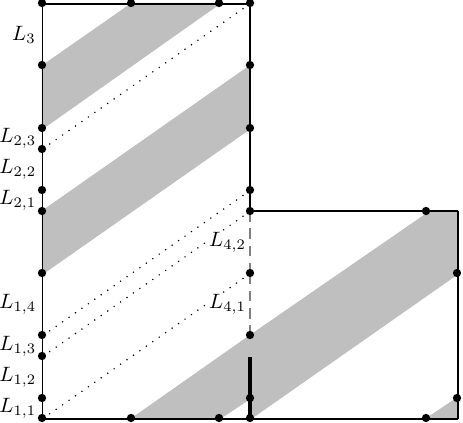}
\\
\mbox{Figure 19: L-surface modified by the inclusion of a one-sided barrier}
\end{array}
\end{displaymath}

Consider the discrete transformation
\begin{displaymath}
\T_\alpha:\VVV\to\VVV
\end{displaymath}
which is the dissipative shift by the vector $(1,\alpha)$.
It follows from Figure~19 that
\begin{equation}\label{eq4.2}
\begin{array}{rl}
\T_\alpha(L_{1,3})=L_{2,1},
&\quad
\T_\alpha(L_{2,3})=L_{1,1},
\\
\T_\alpha(L_{1,4})=L_{2,2}\cup L_{2,3},
&\quad
\T_\alpha(L_{4,1})=L_{1,2}\cup L_{1,3},
\\
\T_\alpha(L_3)=L_{4,1},
&\quad
\T_\alpha(L_{4,2})=L_{1,4},
\\
\T_\alpha(L_{1,1}\cup L_{1,2})=L_{4,2},
&\quad
\T_\alpha(L_{2,1}\cup L_{2,2})=L_3.
\end{array}
\end{equation}
There is a good reason for listing the information in \eqref{eq4.2} in this particular way,
as it is not difficult to see that each column in \eqref{eq4.2} leads to the projections
\begin{equation}\label{eq4.3}
\begin{array}{c}
\T^*([1-\alpha,2\alpha-1))=[0,3\alpha-2),
\\
\T^*([2\alpha-1,\alpha))=[3\alpha-2,2\alpha-1),
\\
\T^*([\alpha,1))=[2\alpha-1,\alpha),
\\
\T^*([0,1-\alpha))=[\alpha,1),
\end{array}
\end{equation}
where $\T^*$ represents the modulo one projection of the second coordinates of~$\T_\alpha$.
It is clear from \eqref{eq4.3} that the images in each column covers the unit interval $[0,1)$ precisely once,
so the unit interval $[0,1)$ is covered twice.

We can also describe the interval exchange transformation $\T_\alpha$ by the directed graph in the picture on the left in Figure~20
or by the undirected graph in the picture on the right in Figure~20,
each constructed using the information given by \eqref{eq4.2}.
In both graphs, the $10$ vertices are the intervals
\begin{equation}\label{eq4.4}
L_{1,1},L_{1,2},L_{1,3},L_{1,4},L_{2,1},L_{2,2},L_{2,3},L_3,L_{4,1},L_{4,2},
\end{equation}
For the graph on the left, a directed edge $L'\to L''$ from a vertex $L'$ to a vertex $L''$
indicates that there exist points $P'\in L'$ and $P''\in L''$ such that $P''=\T_\alpha(P')$.
For the graph on the right, an undirected edge $L'L''$ between vertices $L'$ and $L''$
indicates that there exist points $P'\in L'$ and $P''\in L''$ such that
$P''=\T_\alpha(P')$ or $P'=\T_\alpha(P'')$.
This undirected graph is sometimes known as the \textit{overlapping graph} of the interval exchange transformation.
Clearly this is a connected graph.

\begin{displaymath}
\begin{array}{c}
\includegraphics[scale=0.8]{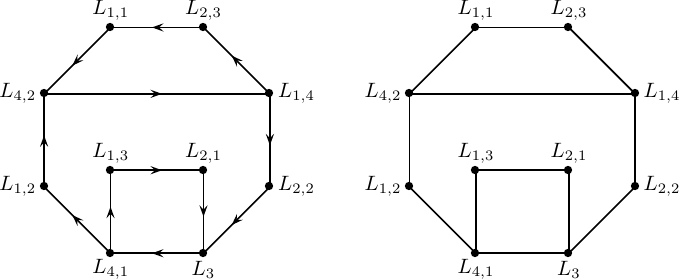}
\\
\mbox{Figure 20: a directed graph and the overlapping graph}
\end{array}
\end{displaymath}

\begin{remark}
Figure~21 shows the recurrent set $\RRR(\PPP;\alpha)$ as a parallelogram
after we have performed some \textit{fake} surgery on the system~$\PPP$, where
we simply take the various pieces of the set and build a parallelogram like a jigsaw puzzle.
Note, however, that the two sides of the dotted line segment between the points $(1,1)$ and $(2,1+\alpha)$ are not neighbours.

\begin{displaymath}
\begin{array}{c}
\includegraphics[scale=0.8]{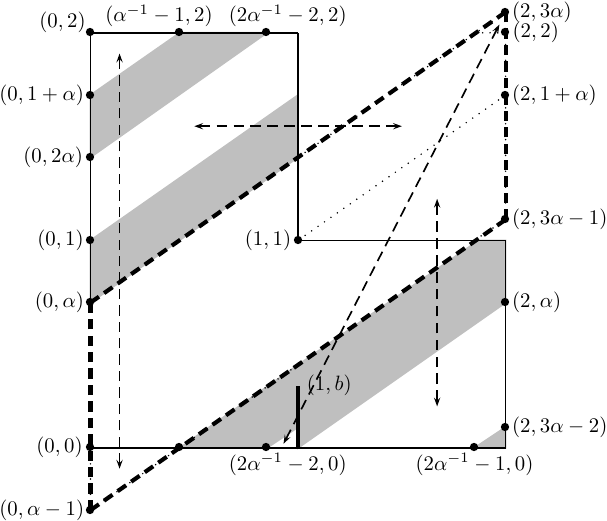}
\\
\mbox{Figure 21: visualization of the recurrent set as a parallelogram}
\end{array}
\end{displaymath}

On the other hand, let us return to Figure~19.
Elementary calculations show that
the part of the recurrent set $\RRR(\PPP;\alpha)$ in the right atomic square has area equal to $2-2\alpha$, 
the part of the recurrent set $\RRR(\PPP;\alpha)$ in the top atomic square has area equal to $2-\alpha^{-1}$, 
while the part of the recurrent set $\RRR(\PPP;\alpha)$ in the bottom left atomic square has area equal to $2\alpha+\alpha^{-1}-2$.
In the special case when $\alpha=1/\sqrt{2}$, these three numbers are respectively roughly equal to $0.586$, $0.586$ and $0.828$,
none of which is close to~$2/3$.
This means that if the dissipative flow of slope $\alpha$ on the recurrent set $\RRR(\PPP;\alpha)$ is uniformly distributed,
then the right atomic square and the top atomic square are \textit{under-visited}, whereas the bottom left atomic square
is \textit{over-visited}.
\end{remark}

Recall that the $2$-cyclic symmetry of the system is a crucial ingredient in our argument
in Section~\ref{sec2} where we consider the modified $2$-square torus.
For the dissipative system $\PPP$ obtained from the L-surface by the inclusion of a one-sided barrier
on the common vertical edge of the bottom atomic squares,
one cannot expect to have any such symmetry property.
However, in the special case when the barrier $B$ and the slope $\alpha$ of the flow
satisfy the special conditions \eqref{eq4.1}, which accounts for uncountably many cases,
we can nevertheless recover some form of symmetry which proves to be very useful.

In Figure~22, the recurrent set $\RRR(\PPP;\alpha)$ of area $2$ is in white
and the transient set $\WWW(\PPP;\alpha)$ of area $1$ is in gray.
Like in Figure~6 for the modified $2$-square torus,
they split the left side of the barrier $B$ and the left side of the corresponding interval $A$ in an essentially fair way.

\begin{displaymath}
\begin{array}{c}
\includegraphics[scale=0.8]{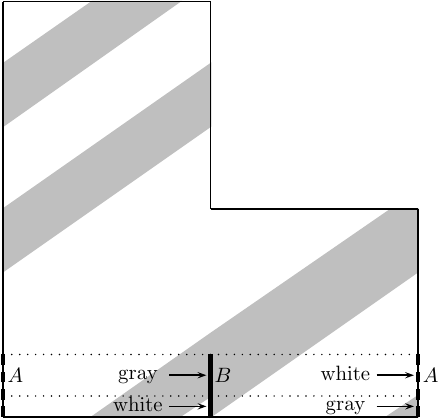}
\\
\mbox{Figure 22: some form of symmetry}
\end{array}
\end{displaymath}

The dissipative flow of slope $\alpha$ moves $B$ first to the gray part in Figure~22
which is $\WWW(\PPP;\alpha)$.
When it reaches the gray part immediately to the left of $B$ or the gray part
immediately to the left of $A$ on right vertical edge, it then jumps to~$A$.
Thus every point of $\WWW(\PPP;\alpha)$ is eventually moved to $A$ by the flow.

Starting from anywhere within the set $\RRR(\PPP;\alpha)$,
the dissipative flow of slope $\alpha$ moves within the white part in Figure~22
which is the set $\RRR(\PPP;\alpha)$.
When it reaches the white part immediately to the left of $B$ or the white part immediately to the left of $A$ on the right vertical edge,
it then \textit{jumps} back to~$A$.
The fair split of left side of the barrier $B$ and the left side of the corresponding interval $A$
implies that this jump is $1$-to-$1$ and measure preserving.
Thus the dissipative flow of slope $\alpha$ restricted to the set $\RRR(\PPP;\alpha)$ is both invertible and measure preserving.
Starting from $A$ and sweeping it along by the dissipative flow of slope~$\alpha$,
this flow returns to $A$ infinitely many times in a $1$-to-$1$ and measure preserving way.

However, since the recurrent set $\RRR(\PPP;\alpha)$ has area~$2$, we cannot simply project to the unit torus $[0,1)^2$
and make use of the integrable geodesic flow there of irrational slope~$\alpha$, as in the case of the modified $2$-square torus.

Observe from \eqref{eq4.2} and \eqref{eq4.3} that the $y$-coordinates of the endpoints of the $10$ subintervals
in \eqref{eq4.4} modulo~$1$ are $\{-\alpha\},0,\{\alpha\},\{2\alpha\},\{3\alpha\}$.

For any irrational number $\alpha>0$ and integer $s\ge2$, suppose that the endpoints of the participating intervals modulo~$1$
of the interval exchange transformation $\T_\alpha:[0,s)\to[0,s)$ have values
\textcolor{white}{xxxxxxxxxxxxxxxxxxxxxxxxxxxxxx}
\begin{equation}\label{eq4.5}
\{-\alpha\},0,\{\alpha\},\{2\alpha\},\ldots,\{\ell_0\alpha\},
\end{equation}
where $\ell_0$ is a fixed positive integer, and that $\T_\alpha$ modulo one acts as the $\alpha$-shift in the unit torus $[0,1)$.
Then the $\ell_0+2$ points in \eqref{eq4.5} divide the unit torus $[0,1)$ into $\ell_0+2$ critical subintervals.
Lifting these critical subintervals in $[0,1)$ to $[0,s)$, we arrive at $s(\ell_0+2)$ critical subintervals in $[0,s)$.
If $S_0\subset[0,s)$ is a non-trivial measurable $\T_\alpha$-invariant subset, then $S_0$ must be a union of some of these
$s(\ell_0+2)$ critical subintervals.

We now use the case $s=2$ and $\ell_0=3$ of the following result which can be proved by a straightforward application of the method
of \cite[Lemma~4.1]{BCY23}.

\begin{lemma}\label{lem41}
If the overlapping graph of the $s(\ell_0+2)$ critical subintervals of the interval exchange transformation $\T_\alpha:[0,s)\to[0,s)$
is connected, then $\T_\alpha$ is ergodic.
Furthermore, $\T_\alpha$ is uniquely ergodic.
\end{lemma}

The special case \eqref{eq4.1} of the modified L-surface turns out to be a lucky case where we can determine
the recurrent set $\RRR(\PPP;\alpha)$ and the transient set $\WWW(\PPP;\alpha)$ without too much difficulty.
The situation may be significantly different if the conditions \eqref{eq4.1} are not satisfied.
Consider the case of the system $\PPP$ as illustrated in Figure~23, where $b=0.8$ and $\alpha=1/\sqrt{2}$.
Here we start from $B$ and sweep out a no-go zone which is the gray polygon with a vertical side on~$B$.
We then extend it by the strip with left edge between $(0,b)$ and $(0,1)$, noting that any flow that reaches
any point between $(1,b)$ and $(1,1)$ eventually reaches~$A$.
We can then further extend the no-go zone to the part of $\PPP$ in gray.

\begin{displaymath}
\begin{array}{c}
\includegraphics[scale=0.8]{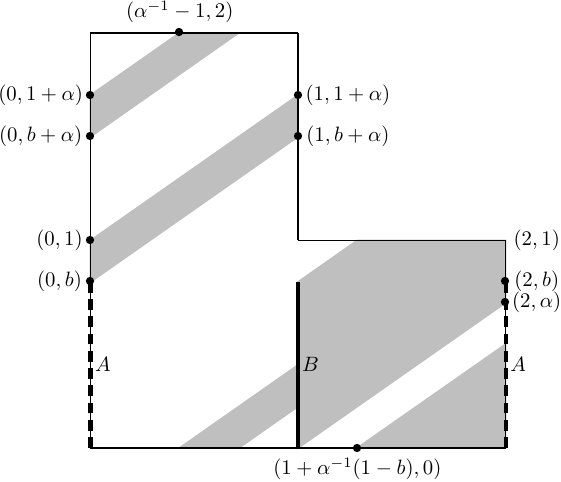}
\\
\mbox{Figure 23: the L-surface modified by the inclusion of a one-sided barrier}
\end{array}
\end{displaymath}

It remains true that any dissipative flow of slope $\alpha$ that reaches the gray part immediately to the left of $B$
or the gray part immediately to the left of $A$ on the right vertical edge continues from immediately to the right of $A$ on the left vertical edge,
but there is no guarantee that this is part of the recurrent set $\RRR(\PPP;\alpha)$.
Figure~24 shows a further extension of the no-go zone.
There is a subinterval of $B$ which is gray immediately to the left and it is also gray immediately to the left
of the corresponding subinterval of~$A$.
This leads to an extension of the no-go zone which is shown in dark gray.
Clearly the set $\frakM(\PPP;B;\alpha)$ is a proper subset of the transient set $\WWW(\PPP;\alpha)$,
and the recurrent set $\RRR(\PPP;\alpha)$ is a proper subset of the set $\frakR(\PPP;A;\alpha)$.

\begin{displaymath}
\begin{array}{c}
\includegraphics[scale=0.8]{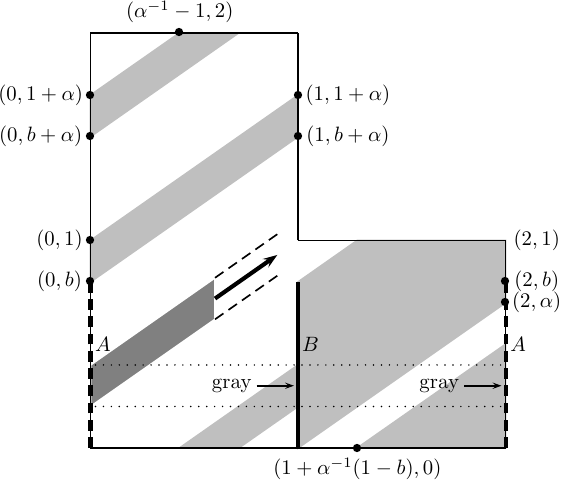}
\\
\mbox{Figure 24: a further extension of the no-go zone}
\end{array}
\end{displaymath}

Another interesting observation is illustrated by Figure~25 which highlights part of the recurrent set $\RRR(\PPP;\alpha)$.

\begin{displaymath}
\begin{array}{c}
\includegraphics[scale=0.8]{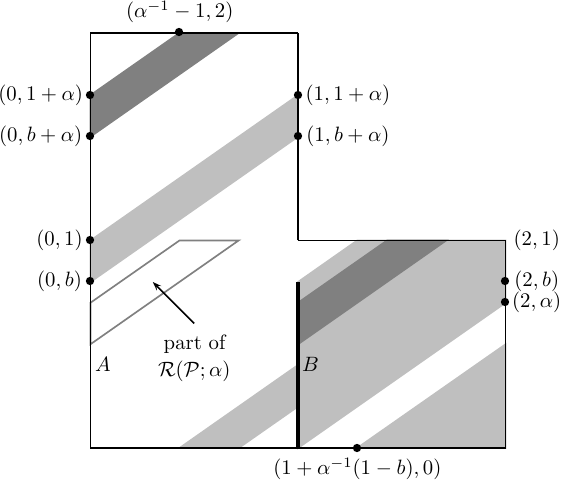}
\\
\mbox{Figure 25: part of the recurrent set}
\end{array}
\end{displaymath}

We see from Figure~24 that both trapezoids in dark gray in Figure~25
are part of the set $\frakM(\PPP;B;\alpha)$, hence part of the transient set $\WWW(\PPP;\alpha)$.
These two trapezoids have the same projection on to the unit torus $[0,1)^2$.
Since the projection of the dissipative flow of slope $\alpha$ to the unit torus $[0,1)^2$ is
integrable geodesic flow of slope~$\alpha$ which is uniformly distributed if $\alpha$ is irrational,
it follows that the white trapezoid in the bottom left atomic square in Figure~25
must be part of the recurrent set $\RRR(\PPP;\alpha)$.

It is clear from Figure~24 that the reverse flow recipe does not work in this case
and does not lead to the recurrent set $\RRR(\PPP;\alpha)$.
However, Figure~25 clearly gives a non-trivial open subset of the recurrent set $\RRR(\PPP;\alpha)$,
so the recurrent set $\RRR(\PPP;\alpha)$ can be obtained by sweeping this open subset along
by the dissipative flow of slope~$\alpha$, and the construction is complete in a finite number of zig-zaggings;
see the Remark after the proof of Theorem~\ref{thm1}.
The recurrent set $\RRR(\PPP;\alpha)$ is then a finite union of polygons with boundary edges that
are vertical or of slope~$\alpha$.

Again, we are lucky in this case, as the recurrent set $\RRR(\PPP;\alpha)$ has area $1$,
and uniformity follows from the uniformity of geodesic flow of slope $\alpha$ in the unit torus $[0,1)^2$
if $\alpha$ is irrational.

We emphasize again that in this case,
the set $\frakM(\PPP;B;\alpha)$ is a proper subset of the transient set $\WWW(\PPP;\alpha)$,
and the recurrent set $\RRR(\PPP;\alpha)$ is a proper subset of the set $\frakR(\PPP;A;\alpha)$.

As for determining the transient set $\WWW(\PPP;\alpha)$, Figure~24 gives
the first step of an extension process through which we can \textit{grow} the dark gray strip.
We discuss this in the next section where we attempt to obtain a generalization of Theorem~\ref{thm1}
to modifications of arbitrary finite polysquare translation surfaces.

Whereas for the $2$-square torus or the L-surface, the choice of the interval $A$ follows naturally from the one-sided barrier~$B$,
this is no longer the case if we start with a polysquare translation surface.

Figure~26 shows that in the case of a translation surface comprising $5$ atomic squares in the form of a cross and where
edge identification comes from perpendicular translation, there is more than one choice for the interval~$A$.
The only requirement is that $A$ lies on a different vertical edge on the same horizontal street that contains~$B$.
If some of the edge identifications do not come from perpendicular translation, the situation looks rather more complicated.
Figure~27 shows two equivalent dissipative systems with the interval $A$ on the vertical edge~$v_2$,
while Figure~28 shows two equivalent dissipative systems with the interval $A$ on the vertical edge~$v_4$.

\begin{displaymath}
\begin{array}{c}
\includegraphics[scale=0.8]{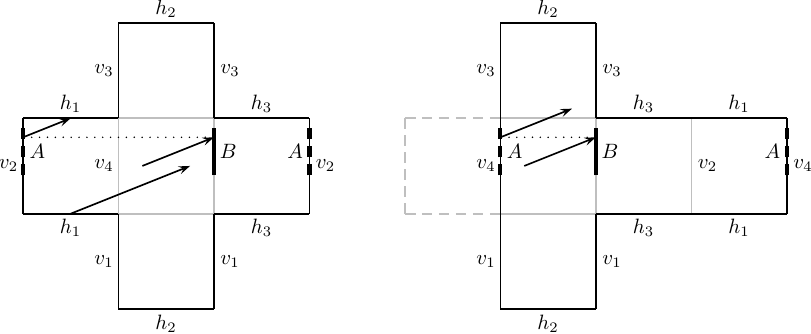}
\\
\mbox{Figure 26: dissipative systems modified from polysquare translation surfaces}
\end{array}
\end{displaymath}
\begin{displaymath}
\begin{array}{c}
\includegraphics[scale=0.8]{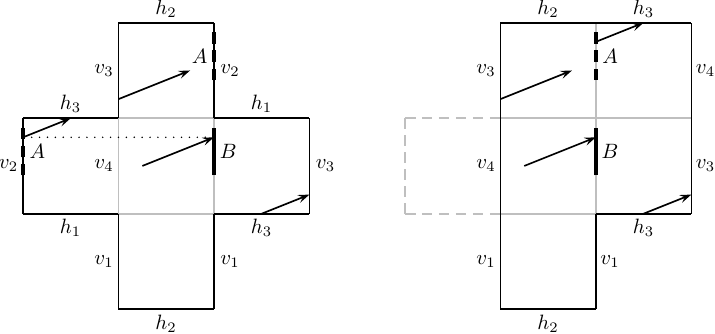}
\\
\mbox{Figure 27: dissipative systems modified from polysquare translation surfaces}
\end{array}
\end{displaymath}
\begin{displaymath}
\begin{array}{c}
\includegraphics[scale=0.8]{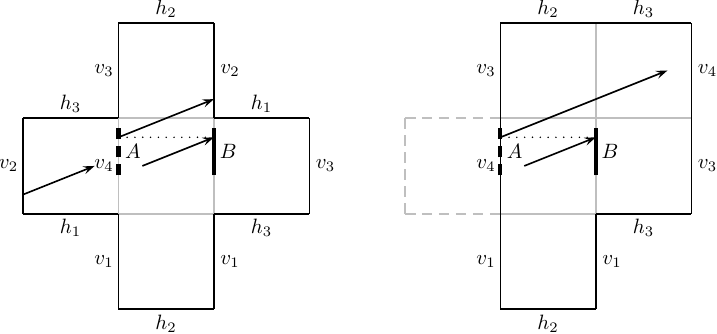}
\\
\mbox{Figure 28: dissipative systems modified from polysquare translation surfaces}
\end{array}
\end{displaymath}

In the next section, we establish the following result.

\begin{theorem}\label{thm4}
Consider dissipative flow of slope~$\alpha$, where $\alpha>0$ is irrational, on a system $\PPP$,
where a finite polysquare translation surface has beem modified with the inclusion of a one-sided barrier $B$ on the common vertical edge
of two neighbouring atomic squares, in the form of a union of finitely many vertical intervals,
and where a set $A$ corresponding to $B$ has been chosen on a different vertical edge on the same horizontal street that contains~$B$.
Then the recurrent set $\RRR(\PPP;\alpha)$ and the transient set $\WWW(\PPP;\alpha)$ can be constructed by an extension process
on the no-go zone $\NNN(\PPP;\alpha)$.
Furthermore, if this process is finite, then both sets $\RRR(\PPP;\alpha)$ and $\WWW(\PPP;\alpha)$ are finite unions of polygons
with boundary edges that are vertical or of slope~$\alpha$.
\end{theorem}

%
%

\section{Extension process and polygons}\label{sec5}

We make use of two well known results in ergodic theory.
Suppose that $(X,\BBB,\mu)$ is a probability space, so that $\mu(X)=1$, and $\T$ is a measure preserving transformation on~$X$.
For any measurable subset $S\subset X$ with positive measure $\mu(S)>0$,
since the Poincare recurrence theorem guarantees infinite recurrence for $\mu$-almost every $x\in S$ to the set~$S$,
it is meaningful to define the \textit{first return time} for any such $x\in S$.
This is given by the integer
\textcolor{white}{xxxxxxxxxxxxxxxxxxxxxxxxxxxxxx}
\begin{displaymath}
n_S(x)=\inf\{n\ge1:\T^n(x)\in S\},
\end{displaymath}
and clearly it is finite for $\mu$-almost every $x\in S$.

The first result is due to Kakutani~\cite{kakutani43} in 1943, and concerns an induced map of $\T$ on~$S$.

\begin{lemma}\label{lem51}
Suppose that $(X,\BBB,\mu)$ is a probability space, so that $\mu(X)=1$,
that $\T$ is an invertible measure preserving transformation on~$X$,
and that $S\subset X$ is a measurable subset with positive measure $\mu(S)>0$.
Then the induced map
\begin{displaymath}
\T\vert_S:S\to S:x\mapsto \T^{n_S(x)}(x),
\end{displaymath}
valid for $\mu$-almost every $x\in S$, is a measure preserving transformation on the induced measure space $(S,\BBB\cap S,\mu_S)$.
\end{lemma}

The second result is due to Kac~\cite{kac47} in 1947, and concerns the \textit{average} value of the first return time to~$S$.

\begin{lemma}\label{lem52}
Suppose that $(X,\BBB,\mu)$ is a probability space, so that $\mu(X)=1$,
that $\T$ is a measure preserving transformation on~$X$,
and that $S\subset X$ is a measurable subset with positive measure $\mu(S)>0$.
Suppose further that $\T$ is ergodic in the space $(X,\BBB,\mu)$.
Then
\textcolor{white}{xxxxxxxxxxxxxxxxxxxxxxxxxxxxxx}
\begin{equation}\label{eq5.1}
\int_Sn_S(x)\,\dd\mu=1.
\end{equation}
In other words, the average value of the first return time to the set $S$ is equal to the reciprocal of the measure $\mu(S)$ of~$S$.
\end{lemma}

\begin{proof}[Proof of Lemma~\ref{lem51}]
Note first of all that both $n_S:S\to\Nn$ and $\T\vert_S:S\to S$ are well defined and measurable mappings.
For every integer $n\ge1$, let
\begin{equation}\label{eq5.2}
S_n=\{x\in S:n_S(x)=n\}
\end{equation}
denote the collection of points $x\in S$ that return to $S$ for the first time at the $n$-th application of the transformation~$\T$.
By definition, the sets $S_n$, $n=1,2,3,\ldots,$ are pairwise disjoint.
Furthermore, we have
\begin{displaymath}
S_1=S\cap \T^{-1}S,
\quad
S_2=(S\cap \T^{-2}S)\setminus S_1,
\quad
S_3=(S\cap \T^{-3}S)\setminus(S_1\cup S_2),
\end{displaymath}
and in general
\textcolor{white}{xxxxxxxxxxxxxxxxxxxxxxxxxxxxxx}
\begin{displaymath}
S_n=(S\cap \T^{-n}S)\setminus\bigcup_{1\le m<n}S_m,
\quad
n=1,2,3,\ldots.
\end{displaymath}

Any measurable subset $A\subset S$ can be expressed as a disjoint union in the form
\begin{displaymath}
A=\bigcup_{n=1}^\infty(A\cap S_n),
\end{displaymath}
so it follows immediately that
\textcolor{white}{xxxxxxxxxxxxxxxxxxxxxxxxxxxxxx}
\begin{equation}\label{eq5.3}
\mu(A)=\sum_{n=1}^\infty\mu(A\cap S_n).
\end{equation}
On the other hand, since $\T$ is invertible and so $1$-to-$1$, the image $\T\vert_SA$ can be expressed as a disjoint union in the form
\begin{displaymath}
\T\vert_SA
=\bigcup_{n=1}^\infty \T\vert_S(A\cap S_n),
\end{displaymath}
so it follows immediately that
\begin{equation}\label{eq5.4}
\mu(\T\vert_SA)
=\sum_{n=1}^\infty\mu(\T\vert_S(A\cap S_n))
=\sum_{n=1}^\infty\mu(\T^n(A\cap S_n)),
\end{equation}
where in the last step, we use the key fact that for every fixed $n=1,2,3,\ldots,$
if a subset $B\subset S_n$ is measurable, then $\T\vert_SB=\T^nB$.
Finally, since $\T$ is invertible and preserves~$\mu$, it follows that
\begin{equation}\label{eq5.5}
\sum_{n=1}^\infty\mu(\T^n(A\cap S_n))=\sum_{n=1}^\infty\mu(A\cap S_n).
\end{equation}
It now follows on combining \eqref{eq5.3}--\eqref{eq5.5} that $\mu(\T\vert_SA)=\mu(A)$, and completes the proof.
\end{proof}

\begin{proof}[Proof of Lemma~\ref{lem52}]
For every $n=1,2,3,\ldots,$ the set $S_n$ defined by \eqref{eq5.2} clearly satisfies
\textcolor{white}{xxxxxxxxxxxxxxxxxxxxxxxxxxxxxx}
\begin{displaymath}
S_n=S_{n,0}=S\cap \T^{-1}S^c\cap\ldots\cap \T^{-(n-1)}S^c\cap \T^{-n}S,
\end{displaymath}
where $S^c=X\setminus S$ denotes the complement of the set $S$ in~$X$.
Furthermore, for every $n=1,2,3,\ldots$ and every $k=1,\ldots,n-1$, write
\begin{equation}\label{eq5.6}
S_{n,k}
=\T^kS_n
=\T^kS\cap \T^{k-1}S^c\cap\ldots\cap \T^{k-n+1}S^c\cap \T^{k-n}S.
\end{equation}
We first show that the sets
\begin{equation}\label{eq5.7}
S_{n,k},
\quad
n=1,2,3,\ldots,
\quad
k=0,1,\ldots,n-1,
\end{equation}
are pairwise disjoint.
To justify this claim, consider first the sets $S_{n,k'}$ and $S_{n,k''}$, where $0\le k'<k''\le n-1$.
Then it follows from \eqref{eq5.6} that
\begin{displaymath}
S_{n,k'}\subset \T^{k'}S
\quad\mbox{and}\quad
S_{n,k''}\subset \T^{k'}S^c,
\end{displaymath}
and so
\textcolor{white}{xxxxxxxxxxxxxxxxxxxxxxxxxxxxxx}
\begin{displaymath}
S_{n,k'}\cap S_{n,k''}\subset \T^{k'}S\cap \T^{k'}S^c=\emptyset.
\end{displaymath}
Suppose next that $n'<n''$, $k'=0,1,\ldots,n'-1$ and $k''=0,1,\ldots,n''-1$.
Then it follows from \eqref{eq5.6} that
\begin{displaymath}
S_{n',k'}\subset \T^{k'}S\cap \T^{k'-n'}S
\quad\mbox{and}\quad
S_{n'',k''}\subset \T^{k''-1}S^c\cap\ldots\cap \T^{k''-n''+1}S^c.
\end{displaymath}
If $k'\le k''-1$, then clearly $S_{n'',k''}\subset \T^{k'}S^c$, and so
\begin{displaymath}
S_{n',k'}\cap S_{n'',k''}\subset \T^{k'}S\cap \T^{k'}S^c=\emptyset.
\end{displaymath}
On the other hand, if $k'>k''-1$, then $-k'\le-k''$.
Since $n'<n''$, it follows that $n'-k'<n''-k''$, and so $k'-n'\ge k''-n''+1$.
Then clearly $S_{n'',k''}\subset \T^{k'-n'}S^c$, and so
\textcolor{white}{xxxxxxxxxxxxxxxxxxxxxxxxxxxxxx}
\begin{displaymath}
S_{n',k'}\cap S_{n'',k''}\subset \T^{k'-n'}S\cap \T^{k'-n'}S^c=\emptyset.
\end{displaymath}
Thus the sets \eqref{eq5.7} are pairwise disjoint as claimed.
We therefore have
\begin{equation}\label{eq5.8}
\mu\left(\bigcup_{n=1}^\infty\bigcup_{k=0}^{n-1}S_{n.k}\right)
=\sum_{n=1}^\infty\sum_{k=0}^{n-1}\mu(S_{n,k}).
\end{equation}
Since $\T$ is measure preserving, it follows from \eqref{eq5.6} that $\mu(S_{n,k})=\mu(S_n)$
for every $k=0,1,\ldots,n-1$.
Combining this with \eqref{eq5.2}, we deduce that
\begin{equation}\label{eq5.9}
\sum_{n=1}^\infty\sum_{k=0}^{n-1}\mu(S_{n,k})
=\sum_{n=1}^\infty n\mu(S_n)
=\sum_{n=1}^\infty\int_{S_n}n_S(x)\,\dd\mu.
\end{equation}
The Poincare recurrence theorem now gives
\begin{displaymath}
\bigcup_{n=1}^\infty S_n=S.
\end{displaymath}
Since this is clearly a pairwise disjoint union, it follows that
\begin{equation}\label{eq5.10}
\sum_{n=1}^\infty\int_{S_n}n_S(x)\,\dd\mu=\int_Sn_S(x)\,\dd\mu.
\end{equation}
Finally, the pairwise disjoint union
\textcolor{white}{xxxxxxxxxxxxxxxxxxxxxxxxxxxxxx}
\begin{displaymath}
\bigcup_{n=1}^\infty\bigcup_{k=0}^{n-1}S_{n.k}
\end{displaymath}
is $\T$-invariant, so it follows from the ergodicity of $\T$ that
\begin{equation}\label{eq5.11}
\mu\left(\bigcup_{n=1}^\infty\bigcup_{k=0}^{n-1}S_{n.k}\right)=\mu(X)=1.
\end{equation}
The desired identity \eqref{eq5.1} now follows on combining \eqref{eq5.8}--\eqref{eq5.11}.
\end{proof}

\begin{proof}[Proof of Theorem~\ref{thm4}]
Let $\PPP$ be a dissipative system described in the statement of the theorem.
Let $X$ denote the union of the left vertical edges of the constituent atomic squares of the underlying polysquare translation surface.
Furthermore, let
\begin{displaymath}
\T=\T_\alpha:X\to X
\end{displaymath}
denote the discretization of the ordinary geodesic flow on the underlying polysquare translation surface.
For every $x\in X$, let $n(x)=n_{A\cup B}(x)$ denote the smallest integer $n\ge1$ such that $\T^n(x)\in A\cup B$.
Then the time-quantitative version of the Gutkin--Veech theorem gives rise to a threshold $N_0=N_0(\PPP;A\cup B;\alpha)$
such that $n(x)\le N_0$ for $\lambda_1$-almost every $x\in X$.

Consider the induced map
\begin{displaymath}
\T^*=\T\vert_{A\cup B}:A\cup B\to A\cup B:x\mapsto \T^{n(x)}(x).
\end{displaymath}
Then it follows from the existence of $N_0=N_0(\PPP;A\cup B;\alpha)$ and Lemma~\ref{lem51} that
$\T^*$ is a finite interval exchange transformation on the set $A\cup B$.
Indeed, the set $A\cup B$ can essentially be decomposed into a finite union of disjoint intervals
such that ordinary geodesic flow of slope $\alpha$ moves each interval splitting-free either
by length $N_0(1+\alpha^2)^{1/2}$ without hitting $A\cup B$ or
by length at most $N_0(1+\alpha^2)^{1/2}$ when returning to $A\cup B$ for the first time.

Let $\FFF(A\cup B)$ denote the first return strip, where $\lambda_1$-almost every point of $A\cup B$
is swept along by the flow of slope $\alpha$ until it hits the set $A\cup B$ again for the first time,
and let $\FFF(A)$ and $\FFF(B)$ denote respectively the corresponding first return strip when we start
with points of $A$ only and with points of $B$ only.

We now apply Lemma~\ref{lem52} with $S=A\cup B$ and $\mu=\lambda_1/s$, where $s$ denotes the number of distinct atomic squares
of the underlying polysquare translation surface, so that $\mu(X)=1$, and deduce that
\begin{displaymath}
\frac{1}{s}\lambda_2(\FFF(A\cup B))=\frac{1}{s}\int_{A\cup B}n(x)\,\dd\lambda_1=1.
\end{displaymath}
Hence $\lambda_2(\FFF(A\cup B))=s$, and so
\begin{displaymath}
\FFF(A)\cup\FFF(B)=\FFF(A\cup B)=\PPP,
\end{displaymath}
apart from exceptional sets of $2$-dimensional Lebesgue measure~$0$.
The reader may wish to check whether this is the same split as obtained by the reverse flow recipe.

Observe that $\lambda_2(\FFF(B))\ge1$.
To see this, we simply observe that the continuous analogue of Lemma~\ref{lem52} applied to the projection of $\FFF(B)$ to the unit torus $[0,1)^2$
shows that the image is the whole unit torus and so has area~$1$.

At this point, we recall that the Kakutani--Kac approach using Lemmas \ref{lem51} and~\ref{lem52}
gives rise to a measure preserving induced map.
We now make use of this.
For every subset $H\subset\PPP$, let
\begin{displaymath}
y(H)=\{y:(x,y)\in\PPP\mbox{ and }(x,y)\in H\mbox{ for some }x\}
\end{displaymath}
denote the collection of the $y$-coordinates of the points of~$H$.
We have two cases, motivated respectively by Figures 22 and~24.

\begin{case1}
We have \textit{perfect} balance, in the sense that
\begin{equation}\label{eq5.12}
\lambda_1(y(\T^*(A)\cap A)\cap y(\T^*(A)\cap B))=0,
\end{equation}
so that the sets $\T^*(A)\cap A$ and $\T^*(A)\cap B$ have essentially no common $y$-coordinates.
Note that
\begin{align}\label{eq5.13}
\lambda_1((\T^*(A)\cap A)\cup(\T^*(A)\cap B))
&
=\lambda_1(\T^*(A)\cap(A\cup B))
\nonumber
\\
&
=\lambda_1(\T^*(A))
=\lambda_1(A)
=\lambda_1(B),
\end{align}
where the last step is a consequence of the measure preserving property of~$\T^*$.
It follows that if \eqref{eq5.12} holds, then
the projection of $\T^*(A)\cap A$ to the unit torus $[0,1)^2$
and the projection of $\T^*(A)\cap B$ to the unit torus $[0,1)^2$
together form the projection of $A$ or $B$ to the unit torus $[0,1)^2$,
as illustrated in Figure~29, where the parts of $\T^*(A)$ in $A$ and $B$
are indicated in bold under $A$ and $B$ respectively.
See also Figure~22 for a concrete example.

\begin{displaymath}
\begin{array}{c}
\includegraphics[scale=0.8]{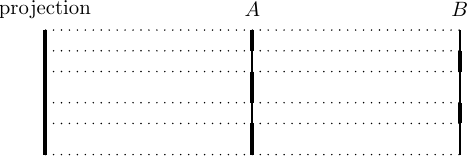}
\\
\mbox{Figure 29: $\T^*(A)\cap A$ and $\T^*(A)\cap B$ are perfectly balanced}
\end{array}
\end{displaymath}

Using the measure preserving property of $\T^*$ again, we conclude that the sets
$\T^*(A)\cap A$ and $\T^*(A)\cap B$ cover $\lambda_1$-almost every $y$-coordinate in $A$ or $B$ precisely once.
This also means that the sets $\T^*(B)\cap A$ and $\T^*(B)\cap B$ cover $\lambda_1$-almost every $y$-coordinate in $A$ or $B$ precisely once.
The dissipative flow of slope $\alpha$ jumps from $\T^*(B)$ to $A$ in a $1$-to-$1$ and measure preserving way,
and jumps from $\T^*(A)$ to $A$ in a $1$-to-$1$ and measure preserving way.

Starting from~$A$, we can extend it by the dissipative flow of slope $\alpha$ forever
in a $1$-to-$1$ and measure preserving way.
Then the resulting set $\frakR(\PPP;A;\alpha)\subset\PPP$ is clearly invariant under dissipative flow of slope~$\alpha$.
It then follows from the Poincare recurrence theorem that $\lambda_2$-almost every point of $\frakR(\PPP;A;\alpha)$
is recurrent under dissipative flow of slope~$\alpha$.
The two sets $\frakR(\PPP;A;\alpha)$ and $\FFF(A)$ have some common properties.
They both contain $A$ and are invariant under dissipative flow of slope~$\alpha$.
They are also both disjoint from $\FFF(B)$.
Thus we conclude that
\begin{equation}\label{eq5.14}
\frakR(\PPP;A;\alpha)=\FFF(A)=\RRR(\PPP;\alpha)
\quad\mbox{and}\quad
\FFF(B)=\WWW(\PPP;\alpha),
\end{equation}
apart from exceptional sets of $2$-dimensional Lebesgue measure~$0$.
If $\RRR(\PPP;\alpha)$ has area~$1$, then we have uniformity of dissipative flow of slope $\alpha$ in the set.
If $\RRR(\PPP;\alpha)$ has area greater than~$1$, then the set may decompose into a union of minimal subsets that
are invariant under dissipative flow of slope $\alpha$ and with integer valued areas.
It then follows from the Birkhoff ergodic theorem that there is uniformity of dissipative flow of slope $\alpha$ in each of these sets.
\end{case1}

\begin{case2}
We do not have perfect balance, in the sense that
\begin{equation}\label{eq5.15}
\lambda_1(y(\T^*(A)\cap A)\cap y(\T^*(A)\cap B))>0,
\end{equation}
so that the sets $\T^*(A)\cap A$ and $\T^*(A)\cap B$ have non-trivial overlapping of $y$-coordinates.
Using the observation \eqref{eq5.13}, we conclude that if \eqref{eq5.15} holds, then
the projection of $\T^*(A)\cap A$ to the unit torus $[0,1)^2$
and the projection of $\T^*(A)\cap B$ to the unit torus $[0,1)^2$
together leave out part of the projection of $A$ or $B$ to the unit torus $[0,1)^2$,
as illustrated in Figure~30, where the parts of $\T^*(A)$ in $A$ and $B$
are indicated in bold under $A$ and $B$ respectively.
See also Figure~24 for a concrete example.

\begin{displaymath}
\begin{array}{c}
\includegraphics[scale=0.8]{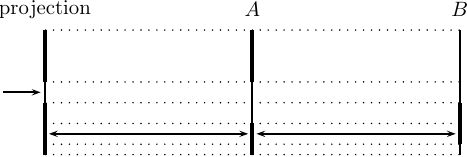}
\\
\mbox{Figure 30: $\T^*(A)\cap A$ and $\T^*(A)\cap B$ are perfectly balanced}
\end{array}
\end{displaymath}

In particular, the dissipative flow of slope $\alpha$ does not jump from $\T^*(A)$ to $A$ in a $1$-to-$1$ and measure preserving way.
It follows that removing $\FFF(B)$ is insufficient to generate dissipative flow of slope $\alpha$ that is measure preserving in the remaining set.
The extension process that we discuss next is now guided by Figure~24.
This is a process that \textit{grows} the transient set in steps.

For the first step, let $A_1\subset A$ be defined by
\begin{displaymath}
y(A_1)=y(\T^*(A)\cap(A\cup B)).
\end{displaymath}
This means that $A_1$ already captures essentially all the $y$-coordinates of $\T^*(A)$, and we throw away the complement $A\setminus A_1$.
More precisely, we expand the set $\FFF(B)$ by including the extra set $\FFF(A\setminus A_1)$.
For the second step, let $A_2\subset A_1$ be defined by
\begin{displaymath}
y(A_2)=y(\T^*(A_1)\cap(A\cup B)).
\end{displaymath}
This means that $A_2$ already captures essentially all the $y$-coordinates of $\T^*(A_1)$, and we throw away the complement $A_1\setminus A_2$.
More precisely, we expand the set $\FFF(B)$ further by including the extra set $\FFF(A_1\setminus A_2)$.
And so on.
We therefore have a decreasing sequence
\textcolor{white}{xxxxxxxxxxxxxxxxxxxxxxxxxxxxxx}
\begin{displaymath}
A\supset A_1\supset A_2\supset\ldots\supset A_i\supset\ldots
\end{displaymath}
of subsets of~$A$.
At the $i$-th step, we expand the set $\FFF(B)$ further by including the extra set $\FFF(A_{i-1}\setminus A_i)$.

Of course, expanding the set $\FFF(B)$ means contracting the set $\FFF(A)$ at the same time.
Accordingly, let
\textcolor{white}{xxxxxxxxxxxxxxxxxxxxxxxxxxxxxx}
\begin{displaymath}
A_\infty=\bigcap_{i=1}^\infty A_i.
\end{displaymath}
Starting from~$A_\infty$, we can extend it by the dissipative flow of slope $\alpha$ forever
in a $1$-to-$1$ and measure preserving way.
Then the resulting set $\frakR(\PPP;A_\infty;\alpha)\subset\PPP$ is clearly invariant under dissipative flow of slope~$\alpha$.
It then follows from the Poincare recurrence theorem that $\lambda_2$-almost every point of $\frakR(\PPP;A_\infty;\alpha)$
is recurrent under dissipative flow of slope~$\alpha$.
The two sets $\frakR(\PPP;A_\infty;\alpha)$ and $\FFF(A_\infty)$ have some common properties.
They both contain $A_\infty$ and are invariant under dissipative flow of slope~$\alpha$.
They are also both disjoint from
\begin{displaymath}
\FFF(B)\cup\bigcup_{i=1}^\infty\FFF(A_{i-1}\setminus A_i),
\end{displaymath}
with the convention that $A_0=A$.
Thus we claim that
\begin{equation}\label{eq5.16}
\frakR(\PPP;A_\infty;\alpha)=\FFF(A_\infty)=\RRR(\PPP;\alpha)
\end{equation}
and
\textcolor{white}{xxxxxxxxxxxxxxxxxxxxxxxxxxxxxx}
\begin{displaymath}
\FFF(B)\cup\bigcup_{i=1}^\infty\FFF(A_{i-1}\setminus A_i)=\WWW(\PPP;\alpha),
\end{displaymath}
apart from exceptional sets of $2$-dimensional Lebesgue measure~$0$.

We need to elaborate on \eqref{eq5.16}.
In Case~1, the sets in \eqref{eq5.14} are finite unions of polygons with boundary edges that are vertical or of slope~$\alpha$.
In this case, we can clearly draw the same conclusion if the extension process is finite.
However, the extension process may possibly be infinite.
In this case, the set $\frakR(\PPP;A_\infty;\alpha)$ may be closed which is much more complicated than a union of polygons.
We then need to show that almost every point in $\frakR(\PPP;A_\infty;\alpha)$ is a non-wandering point.

Note that the set $\frakR(\PPP;A_\infty;\alpha)$ is measurable and invariant under dissipative and measure preserving flow of slope~$\alpha$,
so almost every point $P\in\frakR(\PPP;A_\infty;\alpha)$ satisfies the Lebesgue density theorem,
so that for every $\eps>0$, $P$ has an $\eps$-neighbourhood which intersects $\frakR(\PPP;A_\infty;\alpha)$
in a set of positive measure.
It then follows from the Poincare recurrence theorem that this $\eps$-neighbourhood of $P$ returns via the flow infinitely many times.
Hence $P$ is a non-wandering point.
\end{case2}

This completes the proof.
\end{proof}

Theorem~\ref{thm4} leads immediately to the following question on the finiteness of the extension process.

\begin{question}
What conditions will ensure that the extension process in Theorem~\ref{thm4} terminates after a finite number of steps,
so that the recurrent set $\RRR(\PPP;\alpha)$ is a finite union of polygons with boundary edges that are vertical or of slope~$\alpha$?
Is it true that the extension process always terminates after a finite number of steps?
\end{question}

In the case when the underlying polysquare translation surface is the L-surface, we have already verified this in the affirmative
in the special case of Figure~18 where the conditions \eqref{eq4.1} are satisfied
and in the special case of Figure~23 with $b=0.8$ and $\alpha=1/\sqrt{2}$.

For any typical dissipative system when the underlying polysquare translation surface is the L-surface, we can also establish this
in the affirmative by following the proof of Theorem~\ref{thm4}.
Here, in Case~1, the conclusion is clear.
In Case~2, there must be a non-empty subinterval of $B$ and corresponding subinterval of $A$ that are gray immediately to the left,
leading to a dark gray strip that originates from the subinterval in $A$ analogous to that in Figure~24.
Clearly this dark gray strip has a twin that originates from the subinterval in $B$, and they are both in the transient set $\WWW(\PPP;\alpha)$.
This means that the third twin that originates from the left vertical edge of the atomic square that does not contain $A$ or $B$
must be an open set in the recurrent set $\RRR(\PPP;\alpha)$.
The recurrent set $\RRR(\PPP;\alpha)$ has area~$1$, and its projection to the unit torus $[0,1)^2$ is the whole unit torus.
The desired result now follows from the Remark after the proof of Theorem~\ref{thm1}
concerning a finite number of zig-zaggings.

%
%

\section{More on the extension process}\label{sec6}

We continue our investigation of dissipative systems $\PPP$ where the underlying surface is a finite polysquare translation surface,
as introduced in Section~\ref{sec4}.
We have shown in Section~\ref{sec5} that the extension process, starting from the no-go strip $\NNN(\PPP;\alpha)$,
constructs the transient set $\WWW(\PPP;\alpha)$ in a finite or infinite number of steps.
When the number of steps is finite, then we know that both the transient set $\WWW(\PPP;\alpha)$ and the recurrent set $\RRR(\PPP;\alpha)$
are finite unions of polygons with boundary edges that are vertical or of slope~$\alpha$.
In this section, we establish some partial results concerning the finiteness of the extension process.
A first result is the following.

\begin{theorem}\label{thm5}
In the setting and terminology of Theorem~\ref{thm4}, suppose further that the endpoints of the vertical intervals
that make up the one-sided barrier $B$ all have rational $y$-coordinates.
Then the extension process terminates after a finite number of steps,
and both the transient set $\WWW(\PPP;\alpha)$ and the recurrent set $\RRR(\PPP;\alpha)$
are finite unions of polygons with boundary edges that are vertical or of slope~$\alpha$.
\end{theorem}

\begin{proof}
We follow the proof of Theorem~\ref{thm4} and use the Kakutani--Kac approach.
In Case~1, the extension process does not start at all, so is clearly finite.
In Case~2, we have the situation illustrated in Figure~31.

\begin{displaymath}
\begin{array}{c}
\includegraphics[scale=0.8]{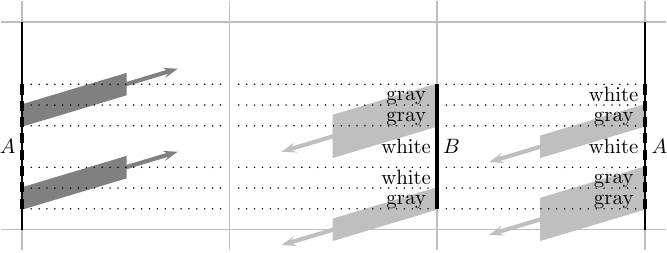}
\\
\mbox{Figure 31: a step of the extension process}
\end{array}
\end{displaymath}

At the start of the $i$-th step of the extension process,  we have the growing no-go zone at the end of the $(i-1)$-th step, coloured in gray,
immediately to the left of $A_{i-1}\setminus A_i$ as well as immediately to the left of the corresponding subintervals of~$B$.
We then construct the set $\FFF(A_{i-1}\setminus A_i)$, which may contain one or more strips depending on the number of intervals $A^\dagger$
that make up the set $A_{i-1}\setminus A_i$.
This is done by starting from each interval $A^\dagger$
and sweeping it along by the flow of slope~$\alpha$ until it hits the set $A\cup B$, as illustrated in dark gray in Figure~31.

This flow from each interval $A^\dagger$ that makes up the set $A_{i-1}\setminus A_i$ may split at a singularity of the system~$\PPP$,
as illustrated in Figure~24.
However, any singularity of $\PPP$ which causes splitting in the $i$-th step is then \textit{eliminated},
in the sense that it will not feature in any subsequent step of the extension process.
Thus all the singularities will be eliminated after finitely many steps of the extension process,
assuming that the extension process has not actually terminated before that.
Hence by ignoring a finite number of steps at the start of the extension process, we may therefore assume without loss of generality that
the flow corresponding to $\FFF(A^\dagger)$ from each interval $A^\dagger$ of $A_{i-1}\setminus A_i$ is splitting-free,
and its image $\T^*(A^\dagger)$ on $A\cup B$ is an interval.
We then need to investigate the location of $\T^*(A^\dagger)$.
The part of $\PPP$ immediately to the left of $\T^*(A^\dagger)$ at the beginning of the $i$-th step is not yet part of the growing
no-go zone and is therefore coloured white.
We study its intersection with $A\cup B$.
For the purpose of our discussion, we assume that $\T^*(A^\dagger)$ lies on~$A$,
as the case if it lies on $B$ is essentially the same, with reference to $A$ and to $B$ interchanged.

For each subset $A_i\subset A$, we also consider the corresponding subset $B_i\subset B$.
We say that an interval $A^\dagger\subset A_i$ is a ww-subinterval (resp. gg-subinterval) of $A_i$ if it is maximal in the sense of inclusion with the property
that $\PPP$ is coloured white (resp. gray) immediately to the left of~$A^\dagger$
and is also coloured white (resp. gray) immediately to the left of the corresponding subinterval $B^\dagger\subset B_i$.
We say that an interval $A^\dagger\subset A_i$ is a wg-subinterval (resp. gw-subinterval) of $A_i$ if it is maximal in the sense of inclusion with the property
that $\PPP$ is coloured white (resp. gray) immediately to the left of~$A^\dagger$
and is coloured gray (resp. white) immediately to the left of the corresponding subinterval $B^\dagger\subset B_i$.
Furthermore, we define the special sum at the start of the $i$-th step of the extension process by
\begin{displaymath}
S(i)=\vert\{A^\dagger\subset A_{i-1}:\mbox{$A^\dagger$ is a ww-subinterval or a gg-subinterval of $A_{i-1}$}\}\vert,
\end{displaymath}
and this represents the total number of ww-subintervals and gg-subintervals of $A_{i-1}$ at the start of the $i$-th step of the extension process.

We need the following intermediate result.

\begin{lemma}\label{lem61}
Suppose that there is no splitting in the $i$-th step of the extension process where the flow takes the set $A_{i-1}\setminus A_i$
to the image $\T^*(A_{i-1}\setminus A_i)$.
Then the inequality $S(i+1)\le S(i)$ holds.
In particular, there exists a positive constant $c_1=c_1(\PPP;B;\alpha)\ge1$ such that $S(i)\le c_1$ for every $i=1,2,3,\ldots.$
\end{lemma}

\begin{proof}
The proof breaks down into a number of cases.

Note first of all that the set $A_{i-1}\setminus A_i$ is the union of all the gg-subintervals $A^\dagger$ of~$A_{i-1}$,
and recall that at the start of the $i$-th step of the extension process, the part of $\PPP$ immediately to
the left of each image $\T^*(A^\dagger)$ is coloured white.
It follows that each image $\T^*(A^\dagger)$ is contained in a union of finitely many ww-subintervals and wg-subintervals of~$A_{i-1}$.

\begin{case1}
Suppose that the image $\T^*(A^\dagger)$ of a gg-subinterval $A^\dagger$ of $A_{i-1}$
lies in the interior of a ww-subinterval $A^\ddagger$ of~$A_{i-1}$,
as illustrated in the picture on the left in Figure~32.
Then after $\FFF(A^\dagger)$ has been added to the growing no-go zone,
this ww-subinterval $A^\ddagger$ of $A_{i-1}$ is converted to $2$ ww-subintervals and a gw-subinterval of~$A_i$.
It follows that going from $S(i)$ to $S(i+1)$, we gain $2$ new ww-subintervals,
lose the ww-subinterval $A^\ddagger$ and lose the gg-subinterval~$A^\dagger$.
Thus $S(i+1)=S(i)$.
\end{case1}

\begin{displaymath}
\begin{array}{c}
\includegraphics[scale=0.8]{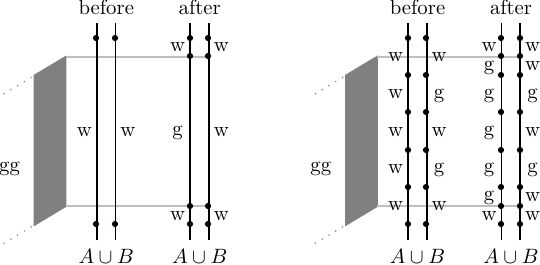}
\\
\mbox{Figure 32: the splitting-free image $\T^*(A^\dagger)$ and consequences}
\end{array}
\end{displaymath}

\begin{case2}
Suppose that the image $\T^*(A^\dagger)$ of a gg-subinterval $A^\dagger$ of $A_{i-1}$
lies in the interior of a union of consecutive ww-subintervals and wg-intervals of~$A_{i-1}$,
with a ww-subinterval of $A_{i-1}$ at each end of the union, as illustrated in the picture on the right in Figure~32.
Then after $\FFF(A^\dagger)$ has been added to the growing no-go zone,
the ww-subinterval of $A_{i-1}$ at the each end is converted to a ww-subinterval and a gw-subinterval of~$A_i$,
each ww-subinterval of $A_{i-1}$ in the middle is converted to a gw-subinterval of~$A_i$,
and each wg-subinterval of $A_{i-1}$ is converted to a gg-subinterval of~$A_i$.
It follows that going from $S(i)$ to $S(i+1)$, we gain one more gg-subinterval of $A_i$ than ww-subinterval of $A_{i-1}$ that we lose,
but we also lose the gg-subinterval~$A^\dagger$.
Thus $S(i+1)=S(i)$.
\end{case2}

\begin{case3}
Suppose that the image $\T^*(A^\dagger)$ of a gg-subinterval $A^\dagger$ of $A_{i-1}$
lies in the interior of a union of consecutive wg-subintervals and ww-intervals of~$A_{i-1}$,
with a wg-subinterval of $A_{i-1}$ at each end of the union, as illustrated in the picture on the left in Figure~33.
Then after $\FFF(A^\dagger)$ has been added to the growing no-go zone,
the wg-subinterval of $A_{i-1}$ at the each end is converted to a gg-subinterval and a wg-subinterval of~$A_i$,
each wg-subinterval of $A_{i-1}$ in the middle is converted to a gg-subinterval of~$A_i$,
and each ww-subinterval of $A_{i-1}$ is converted to a gw-subinterval of~$A_i$.
It follows that going from $S(i)$ to $S(i+1)$, we gain one more gg-subinterval of $A_i$ than ww-subinterval of $A_{i-1}$ that we lose,
but we also lose the gg-subinterval~$A^\dagger$.
Thus $S(i+1)=S(i)$.
\end{case3}

\begin{displaymath}
\begin{array}{c}
\includegraphics[scale=0.8]{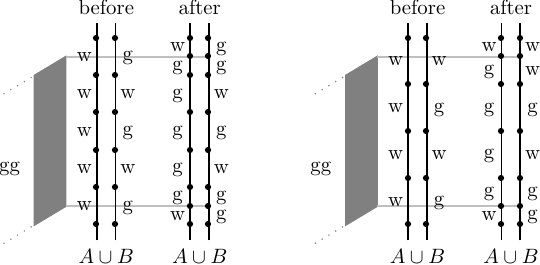}
\\
\mbox{Figure 33: the splitting-free image $\T^*(A^\dagger)$ and consequences}
\end{array}
\end{displaymath}

\begin{case4}
Suppose that the image $\T^*(A^\dagger)$ of a gg-subinterval $A^\dagger$ of $A_{i-1}$
lies in the interior of a union of consecutive wg-subintervals and ww-intervals of~$A_{i-1}$,
with a wg-subinterval of $A_{i-1}$ at one end of the union and a ww-subinterval of $A_{i-1}$ at the other end of the union,
as illustrated in the picture on the right in Figure~33.
Then after $\FFF(A^\dagger)$ has been added to the growing no-go zone,
the wg-subinterval of $A_{i-1}$ at one end is converted to a gg-subinterval and a wg-subinterval of~$A_i$,
the ww-subinterval of $A_{i-1}$ at the other end is converted to a ww-subinterval and a gw-subinterval of~$A_i$,
each wg-subinterval of $A_{i-1}$ in the middle is converted to a gg-subinterval of~$A_i$,
and each ww-subinterval of $A_{i-1}$ in the middle is converted to a gw-subinterval of~$A_i$.
It follows that going from $S(i)$ to $S(i+1)$, we gain one more gg-subinterval of $A_i$ than ww-subinterval of $A_{i-1}$ that we lose,
but we also lose the gg-subinterval~$A^\dagger$.
Thus $S(i+1)=S(i)$.
\end{case4}

We next investigate those instances when an endpoint of the image $\T^*(A^\dagger)$
coincides with an endpoint of the union of consecutive wg-subintervals and ww-intervals of~$A_{i-1}$.

\begin{case5}
Suppose that the extreme endpoint of the wg-subinterval of $A_{i-1}$ at one end of the union of consecutive
wg-subintervals and ww-intervals of $A_{i-1}$ coincides with an endpoint of the interval $\T^*(A^\dagger)$,
as illustrated in the picture on the right in Figure~34.
Compared to the picture on the left, we see that the only difference is that there is now one fewer wg-subinterval of~$A_i$,
and this does not contribute to $S(i+1)$.
Thus $S(i+1)=S(i)$.
\end{case5}

\begin{displaymath}
\begin{array}{c}
\includegraphics[scale=0.8]{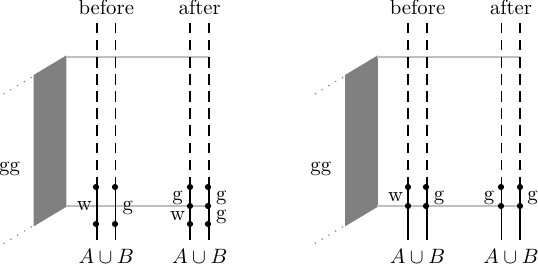}
\\
\mbox{Figure 34: the splitting-free image $\T^*(A^\dagger)$ and consequences}
\end{array}
\end{displaymath}

\begin{case6}
Suppose that the extreme endpoint of the ww-subinterval of $A_{i-1}$ at one end of the union of consecutive
wg-subintervals and ww-intervals of $A_{i-1}$ coincides with an endpoint of the interval $\T^*(A^\dagger)$,
as illustrated in the picture on the right in Figure~35.
Compared to the picture on the left, we see that the only difference is that there is now one fewer ww-subinterval of~$A_i$,
and this makes a difference to $S(i+1)$.
Thus $S(i+1)<S(i)$.
\end{case6}

\begin{displaymath}
\begin{array}{c}
\includegraphics[scale=0.8]{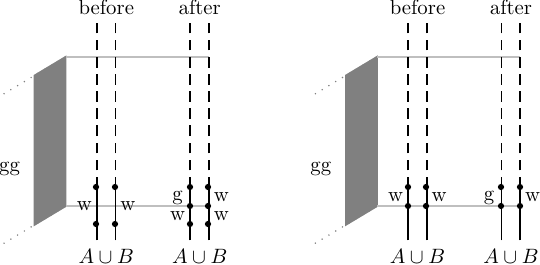}
\\
\mbox{Figure 35: the splitting-free image $\T^*(A^\dagger)$ and consequences}
\end{array}
\end{displaymath}

This completes the proof.
\end{proof}

Assume, on the contrary, that the extension process does not terminate after a finite number of steps.
That means that for every $i=1,2,3,\ldots,$ at the beginning of the $i$-th step, there are gg-subintervals of $A_{i-1}$ to work with.
On the other hand, we see from the proof of Lemma~\ref{lem61} that each gg-subinterval $A^\dagger$ of $A_{i-1}$ originates from
a gg-subinterval of $A_{i-2}$, which in turn originates from a gg-subinterval of $A_{i-3}$, and so on,
and so has its origin from a gg-subinterval of~$A$.
Thus we can examine their combined horizontal flow.
More precisely, for each gg-subinterval $A^\dagger$ of $A_{i-1}$, we consider the total length $L(A^\dagger)$ of the horizontal flow,
starting from a gg-subinterval of $A$ and culminating in the image~$A^\dagger$, as illustrated in Figure~36 where
the flow is stretched out on the plane.

\begin{displaymath}
\begin{array}{c}
\includegraphics[scale=0.8]{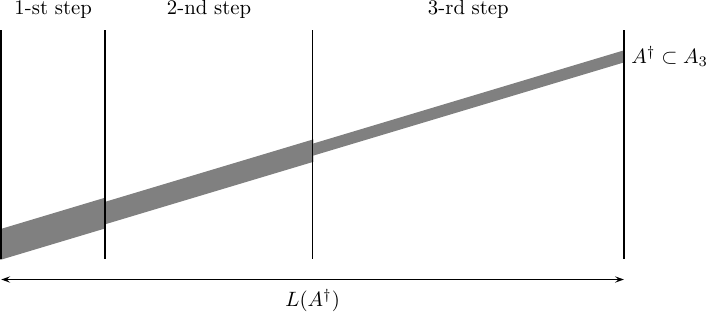}
\\
\mbox{Figure 36: illustration of $L(A^\dagger)$ on the plane}
\end{array}
\end{displaymath}

Summing $L(A^\dagger)$ over the collection $\HHH(i)$ of all the individual gg-subintervals $A^\dagger$ of~$A_{i-1}$ gives rise to the sum
\textcolor{white}{xxxxxxxxxxxxxxxxxxxxxxxxxxxxxx}
\begin{displaymath}
L(i)=\sum_{A^\dagger\in\HHH(i)}L(A^\dagger).
\end{displaymath}
Then the sequence $L(i)$, $i=1,2,3,\ldots,$ is increasing and not bounded above.

We need a well known number theoretic result.

\begin{lemma}\label{lem62}
Let $q_k=q_k(\alpha)$ denote the denominator of the $k$-convergent of the irrational number $\alpha>0$.
Then for every integer $n=1,2,3,\ldots,q_k-1$, we have
\begin{displaymath}
\Vert n\alpha\Vert>\frac{1}{2q_k},
\end{displaymath}
where $\Vert\gamma\Vert$ denotes the distance of $\gamma$ from the nearest integer.
\end{lemma}

Using the time-quantitative version of the Gutkin--Veech theorem and Lemma~\ref{lem61},
we see that there exists a sufficiently large positive constant $c_2=c_2(\PPP;B;\alpha)$
such that any interval of positive real numbers of length at least $c_2$ contains at least one member of the sequence $L(i)$, $i=1,2,3,\ldots.$

Let $q_{k_1},q_{k_2},q_{k_3},\ldots,$ where $k_1<k_2<k_3<\ldots,$ be a sequence of sufficiently large denominators
of the convergents of $\alpha$ such that
\begin{equation}\label{eq6.1}
q_{k_\nu}>2c_2,
\quad
q_{k_{\nu+1}}\ge4c_1q_{k_\nu},
\quad
\nu=1,2,3,\ldots.
\end{equation}
Then there exists an increasing sequence of integers $i_\nu$, $\nu=1,2,3,\ldots,$ such that
\begin{displaymath}
L(i_\nu)\in\left(\frac{q_{k_\nu}}{2},q_{k_\nu}\right),
\quad
\nu=1,2,3,\ldots,
\end{displaymath}
and the intervals are pairwise disjoint, since $c_1\ge1$.
We may further assume that there is no splitting in the extension process from step $i_1$ onwards.

For every $i=1,2,3,\ldots,$ let $\AAAA(i)$ denote the subset of the transient set $\WWW(\PPP;\alpha)$
that the extension process gives after $i$ steps.

We wish to give a lower bound to the the quantities
\begin{displaymath}
\lambda_2(\AAAA(i_\nu)\setminus\AAAA(i_{\nu-1})),
\quad
\nu=2,3,4,\ldots.
\end{displaymath}
In view of Lemma~\ref{lem61}, there exists $A_\nu^\dagger\in\HHH(i_\nu)$ such that
\begin{equation}\label{eq6.2}
L(A_\nu^\dagger)\ge\frac{q_{k_\nu}}{2c_1}.
\end{equation}
To determine the length of the interval~$A_\nu^\dagger$, we use the assumption that all of
the vertical intervals that make up the one-sided barrier $B$ have endpoints with rational $y$-coordinates.
Suppose that $Q=Q(\PPP;B;\alpha)$ denotes the lowest common multiple of the denominators of these $y$-coordinates.
On magnifying the dissipative system $\PPP$ by a factor $Q$ in each direction, we arrive at a dissipative system $\PPP_Q$
where the one-sided barrier corresponding to $B$ consists of whole vertical edges of atomic squares of~$\PPP_Q$.
Thus the magnified analogue of the interval $A_\nu^\dagger$ in $\PPP_Q$ must have endpoints with $y$-coordinates
of the form $\{b+m\alpha\}$ for some $m=0,1,2,3,\ldots,q_{k_\nu}-1$,
where $b\in\RRR$ is fixed, and length
\begin{displaymath}
\vert\{b+m_1\alpha\}-\{b+m_2\alpha\}\vert\ge\Vert(m_1-m_2)\alpha\Vert>\frac{1}{2q_{k_\nu}},
\end{displaymath}
where $0\le m_1<m_2<q_{k_\nu}$, in view of Lemma~\ref{lem62}.
Thus the length of the interval $A_\nu^\dagger$ is at least
\textcolor{white}{xxxxxxxxxxxxxxxxxxxxxxxxxxxxxx}
\begin{displaymath}
\frac{1}{2q_{k_\nu}Q}.
\end{displaymath}
It follows from \eqref{eq6.1} and \eqref{eq6.2} that
\begin{displaymath}
\lambda_2(\AAAA(i_\nu)\setminus\AAAA(i_{\nu-1}))
\ge\frac{L(A_\nu^\dagger)-L(i_{\nu-1})}{2q_{k_\nu}Q}
\ge\frac{1}{2q_{k_\nu}Q}\left(\frac{q_{k_\nu}}{2c_1}-q_{k_{\nu-1}}\right)
\ge\frac{1}{4c_1Q}.
\end{displaymath}

The disjointness of the sets $\AAAA(i_\nu)\setminus\AAAA(i_{\nu-1})$, $\nu=2,3,4,\ldots,$
now implies infinite total area, and this gives the desired contradiction.
\end{proof}

Under the hypotheses of Theorem~\ref{thm5}, the restriction of the dissipative flow of irrational slope $\alpha$ to the
recurrent set $\RRR(\PPP;\alpha)$ modulo one is integrable geodesic flow on the unit torus $[0,1)^2$.
The ergodic theorem then implies that this restricted flow is a $k$-fold covering of geodesic flow on the unit torus $[0,1)^2$
for some integer $k$ satisfying $1<k<s$, where $s$ is the number of atomic squares of
the underlying finite polysquare translation surface of the dissipative system~$\PPP$.
Since the recurrent set $\RRR(\PPP;\alpha)$ is a finite union of polygons with boundary edges that are vertical or of slope~$\alpha$,
the usual discretization of the flow restricted to $\RRR(\PPP;\alpha)$ defines a bijective interval exchange transformation
from the interval $[0,k)$ to itself.
If the overlapping graph of this bijection is connected, then Lemma~\ref{lem41} guarantees equidistribution
of the orbits inside $\RRR(\PPP;\alpha)$.
In this case, we say that the attractor $\RRR(\PPP;\alpha)$ is \textit{minimal}.

We also establish a second partial result concerning the finiteness of the extension process.

\begin{theorem}\label{thm6}
In the setting and terminology of Theorem~\ref{thm4}, suppose further that a component of the recurrent set $\RRR(\PPP;\alpha)$
contains an open subset~$G$.
Then this component of the recurrent set $\RRR(\PPP;\alpha)$ can be constructed by a finite number of zig-zaggings
and is a finite union of polygons with boundary edges that are vertical or of slope~$\alpha$.
\end{theorem}

Note that we have no information on the structure of the part of the recurrent set $\RRR(\PPP;\alpha)$ outside this special component,
so we are unable to draw any conclusions about the recurrent set $\RRR(\PPP;\alpha)$ as a whole.

We begin with some general facts.

Suppose that the underlying polysquare translation surface of the dissipative system $\PPP$ has $s$ atomic squares.
Then the dissipative flow of irrational slope $\alpha>0$ on $\PPP$ defines a non-bijective transformation
\begin{equation}\label{eq6.3}
f=f_\alpha:[0,s)\to[0,s),
\end{equation}
where the interval $[i-1,i)$ represents the left vertical edge of the $i$-th atomic square.

\begin{remark}
Strictly speaking, $f=f_\alpha$ is not an interval exchange transformation as dissipative flow is in general not reversible.
\end{remark}

Denote modulo one projection by
\begin{equation}\label{eq6.4}
\pi:[0,s)\to[0,1).
\end{equation}
The function $\pi$ relates the discretization of dissipative flow of slope $\alpha$ on $\PPP$
to the discretization of geodesic flow of slope $\alpha$ on the unit torus $[0,1)^2$.

Let $\SSSS\subset[0,s)$ be a subset of multiplicity~$1$.
In other words, the restriction
\begin{displaymath}
\pi_\SSSS:\SSSS\to[0,1):y\mapsto\pi(y)
\end{displaymath}
of $\pi$ to the set $\SSSS$ is injective.
We further assume that $\SSSS$ is a union of finitely many intervals, so likewise for the set
\begin{displaymath}
S=\pi(\SSSS).
\end{displaymath}

Let the modulo one projection of $f=f_\alpha$ be denoted by
\begin{equation}\label{eq6.5}
g=g_\alpha:[0,1)\to[0,1),
\end{equation}
so that
\textcolor{white}{xxxxxxxxxxxxxxxxxxxxxxxxxxxxxx}
\begin{equation}\label{eq6.6}
g\circ\pi=\pi\circ f.
\end{equation}
Then $g=g_\alpha$ describes rotation by $\alpha$ on the unit circle $[0,1)$,
and is the interval exchange transformation corresponding to the discretization of geodesic flow
of slope $\alpha$ on the unit torus $[0,1)^2$.

For any $x\in S\subset[0,1)$, let
\begin{displaymath}
n(x)=n_S(x)=\min\{n\ge1:g^n(x)\in S\}.
\end{displaymath}
Then it follows from Lemma~\ref{lem52} due to Kac that there exists a common threshold $N^*(S)=N^*(S;\alpha)$ such that
\begin{displaymath}
n(x)=n_S(x)\le N^*(S)
\quad\mbox{for all $x\in S$}.
\end{displaymath}
As a consequence of \eqref{eq5.11} in the proof of Lemma~\ref{lem52}, we have
\begin{equation}\label{eq6.7}
\!\!\!
[0,1)
=\bigcup_{n=1}^{N^*(S)}\bigcup_{k=0}^{n-1}g^k(\{x\in S:n_S(x)=n\})
=\bigcup_{n=1}^{N^*(S)}\bigcup_{k=1}^ng^k(\{x\in S:n_S(x)=n\}),
\end{equation}
apart from an exceptional set of measure~$0$ and where each union is pairwise disjoint.
We then lift this up to $[0,m)$ and obtain a quantity
\begin{equation}\label{eq6.8}
\Psi(\SSSS)=\bigcup_{n=1}^{N^*(S)}\bigcup_{k=1}^nf^k(\{y\in\SSSS:n_S(\pi(y))=n\}),
\end{equation}
and of course $\pi(\Psi(\SSSS))=[0,1)$.

The assumption that $\SSSS$ is a set of multiplicity $1$ implies that $\pi$ is injective on $\Psi(\SSSS)$;
in other words, $\Psi(\SSSS)$ is a copy of the unit torus $[0,1)$ generated from $\SSSS$ by the dissipative flow of slope~$\alpha$.
More importantly, $\Psi(\SSSS)$ is a set of multiplicity~$1$.

The set
\textcolor{white}{xxxxxxxxxxxxxxxxxxxxxxxxxxxxxx}
\begin{displaymath}
\phi(\SSSS)=\bigcup_{n=1}^{N^*(S)}f^n(\{y\in\SSSS:n_S(\pi(y))=n\})
\end{displaymath}
is essentially a copy~$\SSSS$, since
\begin{align}
\pi(\phi(\SSSS))
&
=\bigcup_{n=1}^{N^*(S)}\pi(f^n(\{y\in\SSSS:n_S(\pi(y))=n\}))
\nonumber
\\
&
=\bigcup_{n=1}^{N^*(S)}g^n(\pi(\{y\in\SSSS:n_S(\pi(y))=n\}))
\nonumber
\\
&
=\bigcup_{n=1}^{N^*(S)}g^n(\{x\in S:n_S(x)=n\})
\nonumber
\\
&
=S=\pi(\SSSS),
\nonumber
\end{align}
apart from an exceptional set of measure~$0$, as it follows from \eqref{eq6.7} that
\begin{displaymath}
\bigcup_{n=1}^{N^*(S)}g^n(\{x\in S:n_S(x)=n\})=\bigcup_{n=1}^{N^*(S)}\{x\in S:n_S(x)=n\}=S,
\end{displaymath}
apart from an exceptional set of measure~$0$.

Clearly $\phi(\SSSS)\subset\Psi(\SSSS)$, so $\phi(\SSSS)$ is a set of multiplicity~$1$.

We next set up the crucial step in the proof of Theorem~\ref{thm6}.

The open subset $G$ of $\PPP$ contains a disc of positive radius.
We may assume that this disc is sufficiently small that it is contained in a single atomic square of the underlying polysquare translation surface of $\PPP$
and the reverse flow in the direction $(-1,-\alpha)$ projects it on to the left vertical edge of a single atomic square,
as illustrated in Figure~37.

\begin{displaymath}
\begin{array}{c}
\includegraphics[scale=0.8]{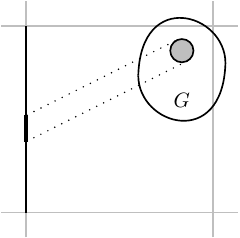}
\\
\mbox{Figure 37: constructing the interval $\IIII_0$}
\end{array}
\end{displaymath}

This gives rise to an interval $\IIII_0\subset[0,s)$ of positive length and multiplicity $1$
on the left vertical edge of an atomic square of the underlying polysquare translation surface of~$\PPP$.

Define the sets $\IIII_1,\IIII_2,\IIII_3,\ldots$ in terms of $\IIII_0$ inductively by
\begin{equation}\label{eq6.9}
\IIII_1=\phi(\IIII_0)\setminus\IIII_0, 
\quad
\IIII_j=\phi(\IIII_{j-1})\setminus\IIII_0,
\quad
j=2,3,4,\ldots,
\end{equation}
assuming that $\IIII_{j-1}$ has positive measure.
Note that each set $\IIII_1,\IIII_2,\IIII_3,\ldots$ is of multiplicity~$1$, and each $\Psi(\IIII_j)$ is a copy of $[0,1)$ if $\IIII_j$ has positive measure.

\begin{lemma}\label{lem63}
For any non-negative integers $j'\ne j''$, the sets $\Psi(\IIII_{j'})$ and $\Psi(\IIII_{j''})$ are disjoint.
\end{lemma}

\begin{proof}[Proof of Theorem~\ref{thm6}]
By Lemma~\ref{lem63}, the $s+1$ sets
\begin{displaymath}
\Psi(\IIII_j)\subset[0,s),
\quad
j=0,1,2,3,\ldots,s,
\end{displaymath}
are pairwise disjoint.
Note next that $\Psi(\IIII_j)$ is a copy of $[0,1)$, and so has measure~$1$, if $\IIII_j$ has positive measure.
This implies that $\Psi(\IIII_{j^*})$ must have measure~$0$ for some $j^*=0,1,2,3,\ldots,s$.
Hence the set
\textcolor{white}{xxxxxxxxxxxxxxxxxxxxxxxxxxxxxx}
\begin{equation}\label{eq6.10}
\bigcup_{j=0}^{j^*-1}\Psi(\IIII_j)
\end{equation}
is generated from $\IIII_0$ in at most
\textcolor{white}{xxxxxxxxxxxxxxxxxxxxxxxxxxxxxx}
\begin{displaymath}
\sum_{j=0}^{j^*-1}N^*(\pi(\IIII_j))
\end{displaymath}
zig-zaggings.
Clearly the set \eqref{eq6.10} is invariant under~$f$.
Note that each set $\Psi(\IIII_j)$ for any integer $j\ge j^*$ has measure $0$ and so can be ignored.
\end{proof}

\begin{remark}
Figure~38 illustrates how we set up an induction process and explains the definitions in \eqref{eq6.9}.
It also explains the purpose of Lemma~\ref{lem63}.

\begin{displaymath}
\begin{array}{c}
\includegraphics[scale=0.8]{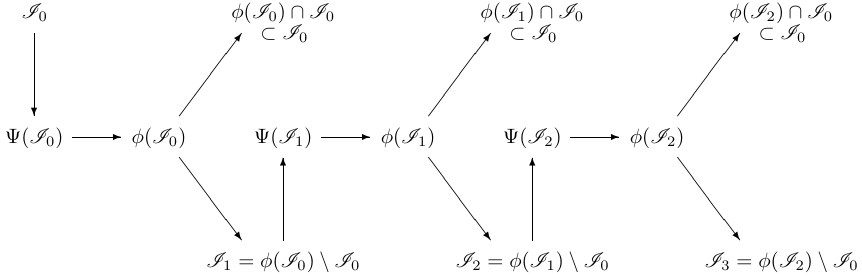}
\\
\mbox{Figure 38: the early steps of the induction process}
\end{array}
\end{displaymath}

The idea is as follows.
Starting from the set~$\IIII_0$, we generate the set $\Psi(\IIII_0)$ by the flow which is represented in discrete form by~$f$.
This set $\Psi(\IIII_0)$ has an interesting subset $\phi(\IIII_0)$ which collects together all the points corresponding to first returns to $\IIII_0$
modulo one.
In general, $\phi(\IIII_0)$ has a part that lies in~$\IIII_0$ and a part that lies outside~$\IIII_0$.
Since the part in $\IIII_0$ is already accounted for, we remove it from $\phi(\IIII_0)$ and consider the subset $\IIII_1=\phi(\IIII_0)\setminus\IIII_0$.
Next, we generate a set $\Psi(\IIII_1)$ which has an interesting subset~$\phi(\IIII_1)$.
Again, $\phi(\IIII_1)$ has a part that lies in~$\IIII_0$ and a part that lies outside~$\IIII_0$.
Hence we consider the subset $\IIII_2=\phi(\IIII_1)\setminus\IIII_0$.
And so on.
In this way, we ensure that the sets $\Psi(\IIII_0),\Psi(\IIII_1),\Psi(\IIII_2),\ldots$ are pairwise disjoint.
Together they take up all the points in $[0,s)$ corresponding to the component of the recurrent set under consideration.
\end{remark}

\begin{proof}[Proof of Lemma~\ref{lem63}]
Note that the lemma follows immediately from the uniqueness property of Assertion~A below.

\begin{assertiona}
For every integer $\ell\ge0$ and $y^*\in\Psi(\IIII_\ell)$, there is a unique sequence
\begin{displaymath}
(y_0,h_0),
\quad
(y_1,h_1),
\quad
(y_2,h_2),
\quad
\ldots,
\quad
(y_\ell,h_\ell),
\end{displaymath}
where $y_j\in\IIII_j$ and $h_j=1,2,3,\ldots$ for every $j=0,1,\ldots,\ell$, such that
\begin{displaymath}
y_1=f^{h_0}(y_0),
\quad
y_2=f^{h_1}(y_1),
\quad
\ldots,
\quad
y_\ell=f^{h_{\ell-1}}(y_{\ell-1}),
\quad
y^*=f^{h_\ell}(y_\ell).
\end{displaymath}
\end{assertiona}

Assertion~A follows from Assertion~B below by backtracking from $y^*$ to~$y_0$,
since the set $\IIII_0$ is never visited in between.

\begin{assertionb}
For every integer $\ell\ge0$ and $y^*\in\Psi(\IIII_\ell)$, there is a sequence
\begin{displaymath}
(y_0,h_0),
\quad
(y_1,h_1),
\quad
(y_2,h_2),
\quad
\ldots,
\quad
(y_\ell,h_\ell),
\end{displaymath}
where $y_j\in\IIII_j$ and $h_j=1,2,3,\ldots$ for every $j=0,1,\ldots,\ell$, such that the numbers
\begin{equation}\label{eq6.11}
\begin{array}{lllll}
y_0,
&\quad f(y_0),
&\quad f^2(y_0),
&\quad \ldots,
&\quad f^{h_0-1}(y_0),
\vspace{3pt}\\
y_1=f^{h_0}(y_0),
&\quad f(y_1),
&\quad f^2(y_1),
&\quad \ldots,
&\quad f^{h_1-1}(y_1),
\vspace{3pt}\\
y_2=f^{h_1}(y_1),
&\quad f(y_2),
&\quad f^2(y_2),
&\quad \ldots,
&\quad f^{h_2-1}(y_2),
\vspace{3pt}\\
&&\quad\ldots,
\vspace{3pt}\\
y_\ell=f^{h_{\ell-1}}(y_{\ell-1}),
&\quad f(y_\ell),
&\quad f^2(y_\ell),
&\quad \ldots,
&\quad f^{h_\ell-1}(y_\ell),
\vspace{3pt}\\
y^*=f^{h_\ell}(y_\ell),
\end{array}
\end{equation}
apart possibly from $y_0$ and~$y^*$, do not belong to~$\IIII_0$.
\end{assertionb}

\begin{remark}
In Assertion~A and Assertion~B, for every $j=0,1,\ldots,\ell-1,$ we have
\begin{displaymath}
f^{h_j}(y_j)=y_{j+1}\in\IIII_{j+1}\subset\phi(\IIII_j),
\quad\mbox{so that}\quad
h_j=n_{\pi(\IIII_j)}(\pi(y_j)).
\end{displaymath}
On the other hand, we have
\begin{displaymath}
f^{h_\ell}(y_\ell)\in\Psi(\IIII_\ell),
\quad\mbox{so that}\quad
h_\ell\le n_{\pi(\IIII_\ell)}(\pi(y_\ell)).
\end{displaymath}
\end{remark}

It remains to establish Assertion~B.
We proceed by induction on~$\ell$.

Consider first the initial case $\ell=0$, and let $y^*\in\Psi(\IIII_0)$.
It follows from \eqref{eq6.8} that there exist $n=1,\ldots,N^*(\pi(\IIII_0))$ and $k=1,\ldots,n$ such that
\begin{displaymath}
y^*\in f^k(\{y\in\IIII_0:n_{\pi(\IIII_0)}(\pi(y))=n\}),
\end{displaymath}
so that $y^*=f^k(y_0)$ for some $y_0\in\IIII_0$ satisfying $n_{\pi(\IIII_0)}(\pi(y_0))=n$, and so \eqref{eq6.11} for $\ell=0$ is satisfied with $h_0=k\le n$
and none of $f(y_0),f^2(y_0),\ldots,f^{k-1}(y_0)$ belongs to~$\IIII_0$.

Suppose next that $i>0$ and Assertion~B holds for every $\ell<i$.
Let $y^*\in\Psi(\IIII_i)$.
It follows from \eqref{eq6.8} that there exist $n=1,\ldots,N^*(\pi(\IIII_i))$ and $k=1,\ldots,n$ such that
\textcolor{white}{xxxxxxxxxxxxxxxxxxxxxxxxxxxxxx}
\begin{displaymath}
y^*\in f^k(\{y\in\IIII_i:n_{\pi(\IIII_i)}(\pi(y))=n\}),
\end{displaymath}
so that $y^*=f^k(y_i)$ for some $y_i\in\IIII_i$ satisfying $n_{\pi(\IIII_i)}(\pi(y_i))=n$.
Write $h_i=k\le n$.
Then we have a pair $(y_i,h_i)$ where $y_i\in\IIII_i$ and $h_i=1,2,3,\ldots,$ leading to the numbers
\textcolor{white}{xxxxxxxxxxxxxxxxxxxxxxxxxxxxxx}
\begin{equation}\label{eq6.12}
\begin{array}{lllll}
y_i,
&\quad f(y_i),
&\quad f^2(y_i),
&\quad \ldots,
&\quad f^{h_i-1}(y_i),
\vspace{3pt}\\
y^*=f^{h_i}(y_i).
\end{array}
\end{equation}
Since $y_i\in\IIII_i\subset\Psi(\IIII_{i-1})$, using the case $\ell=i-1$, there is a sequence
\begin{displaymath}
(y_0,h_0),
\quad
(y_1,h_1),
\quad
(y_2,h_2),
\quad
\ldots,
\quad
(y_{i-1},h_{i-1}),
\end{displaymath}
where $y_j\in\IIII_j$ and $h_j=1,2,3,\ldots$ for every $j=0,1,2,\ldots,i-1$, such that the numbers
\begin{equation}\label{eq6.13}
\begin{array}{lllll}
y_0,
&\quad f(y_0),
&\quad f^2(y_0),
&\quad \ldots,
&\quad f^{h_0-1}(y_0),
\vspace{3pt}\\
y_1=f^{h_0}(y_0),
&\quad f(y_1),
&\quad f^2(y_1),
&\quad \ldots,
&\quad f^{h_1-1}(y_1),
\vspace{3pt}\\
y_2=f^{h_1}(y_1),
&\quad f(y_2),
&\quad f^2(y_2),
&\quad \ldots,
&\quad f^{h_2-1}(y_2),
\vspace{3pt}\\
&&\quad\ldots,
\vspace{3pt}\\
y_{i-1}=f^{h_{i-2}}(y_{i-2}),
&\quad f(y_{i-1}),
&\quad f^2(y_{i-1}),
&\quad \ldots,
&\quad f^{h_{i-1}-1}(y_{i-1}),
\vspace{3pt}\\
y_i=f^{h_{i-1}}(y_{i-1}),
\end{array}
\end{equation}
apart possibly from $y_0$ and~$y_i$, do not belong to~$\IIII_0$.
Clearly \eqref{eq6.13} combines with \eqref{eq6.12} to give the existence aspect of \eqref{eq6.11} for the case $\ell=i$.
It remains to show that the numbers in \eqref{eq6.12}, apart possibly from~$y^*$, do not belong to~$\IIII_0$.

Clearly $y_i\in\IIII_i=\phi(\IIII_{i-1})\setminus\IIII_0$ does not belong to~$\IIII_0$.
Suppose on the contrary that some other number in \eqref{eq6.12}, apart from $y_i$ and~$y^*$, belongs to~$\IIII_0$.
Then it is of the form
\textcolor{white}{xxxxxxxxxxxxxxxxxxxxxxxxxxxxxx}
\begin{displaymath}
f^h(y_i)
\quad
\mbox{for some $h=1,\ldots,h_i-1$},
\end{displaymath}
and is clearly in $\Psi(\IIII_i)$ but not in~$\phi(\IIII_i)$.
Let $y'\in\IIII_0$ denote the one with minimal value of~$h$.
Then
\textcolor{white}{xxxxxxxxxxxxxxxxxxxxxxxxxxxxxx}
\begin{equation}\label{eq6.14}
y'\in\IIII_0
\quad\mbox{and}\quad
y'\not\in\phi(\IIII_i).
\end{equation}
We can next truncate \eqref{eq6.12} and replace it with
\begin{equation}\label{eq6.15}
\begin{array}{lllll}
y_i,
&\quad f(y_i),
&\quad f^2(y_i),
&\quad \ldots,
&\quad f^{h-1}(y_i),
\vspace{3pt}\\
y'=f^h(y_i).
\end{array}
\end{equation}
Combining \eqref{eq6.13} and \eqref{eq6.15} leads to
\begin{equation}\label{eq6.16}
\begin{array}{lllll}
y_0,
&\quad f(y_0),
&\quad f^2(y_0),
&\quad \ldots,
&\quad f^{h_0-1}(y_0),
\vspace{3pt}\\
y_1=f^{h_0}(y_0),
&\quad f(y_1),
&\quad f^2(y_1),
&\quad \ldots,
&\quad f^{h_1-1}(y_1),
\vspace{3pt}\\
y_2=f^{h_1}(y_1),
&\quad f(y_2),
&\quad f^2(y_2),
&\quad \ldots,
&\quad f^{h_2-1}(y_2),
\vspace{3pt}\\
&&\quad\ldots,
\vspace{3pt}\\
y_{i-1}=f^{h_{i-2}}(y_{i-2}),
&\quad f(y_{i-1}),
&\quad f^2(y_{i-1}),
&\quad \ldots,
&\quad f^{h_{i-1}-1}(y_{i-1}),
\vspace{3pt}\\
y_i=f^{h_{i-1}}(y_{i-1}),
&\quad f(y_i),
&\quad f^2(y_i),
&\quad \ldots,
&\quad f^{h-1}(y_i),
\vspace{3pt}\\
y'=f^h(y_i),
\end{array}
\end{equation}
where, apart possibly from $y_0$ and~$y'$, the terms do not belong to~$\IIII_0$.

Meanwhile, it follows from \eqref{eq6.14} that
\begin{displaymath}
\pi(y')\in\pi(\IIII_0)\setminus\pi(\phi(\IIII_i))
=\pi(\IIII_0)\setminus\pi(\IIII_i).
\end{displaymath}
It is easy to check that $\pi(\IIII_j)\subset\IIII_{j-1}$ for every $j=1,2,3,\ldots,$ so there exists some $\kappa<i$ such that
\textcolor{white}{xxxxxxxxxxxxxxxxxxxxxxxxxxxxxx}
\begin{displaymath}
\pi(y')\in\pi(\IIII_\kappa)
\quad\mbox{and}\quad
\pi(y')\not\in\pi(\IIII_{\kappa+1}).
\end{displaymath}
Now
\textcolor{white}{xxxxxxxxxxxxxxxxxxxxxxxxxxxxxx}
\begin{align}
\pi(\IIII_{\kappa+1})
&
=\pi(\phi(\IIII_\kappa)\setminus\IIII_0)
=\pi(\phi(\IIII_\kappa))\setminus\pi(\phi(\IIII_\kappa)\cap\IIII_0)
\nonumber
\\
&
=\pi(\IIII_\kappa)\setminus\pi(\phi(\IIII_\kappa)\cap\IIII_0),
\nonumber
\end{align}
so that $\pi(y')\in\pi(\phi(\IIII_\kappa)\cap\IIII_0)$.
Together with the assumption $y'\in\IIII_0$, this leads to
\begin{displaymath}
y'\in\phi(\IIII_\kappa)\subset\Psi(\IIII_\kappa).
\end{displaymath}
The case $\ell=\kappa$ of the assertion now gives a sequence
\begin{displaymath}
(y^\star_0,h^\star_0),
\quad
(y^\star_1,h^\star_1),
\quad
(y^\star_2,h^\star_2),
\quad
\ldots,
\quad
(y^\star_\kappa,h^\star_\kappa),
\end{displaymath}
where $y^\star_k\in\IIII_k$ and $h^\star_k=1,2,3,\ldots$ for every $k=0,1,2,\ldots,\kappa$, and the numbers
\begin{equation}\label{eq6.17}
\begin{array}{lllll}
y^\star_0,
&\quad f(y^\star_0),
&\quad f^2(y^\star_0),
&\quad \ldots,
&\quad f^{h^\star_0-1}(y^\star_0),
\vspace{3pt}\\
y^\star_1=f^{h^\star_0}(y^\star_0),
&\quad f(y^\star_1),
&\quad f^2(y^\star_1),
&\quad \ldots,
&\quad f^{h^\star_1-1}(y^\star_1),
\vspace{3pt}\\
y^\star_2=f^{h^\star_1}(y^\star_1),
&\quad f(y^\star_2),
&\quad f^2(y^\star_2),
&\quad \ldots,
&\quad f^{h^\star_2-1}(y^\star_2),
\vspace{3pt}\\
&&\quad\ldots,
\vspace{3pt}\\
y^\star_\kappa=f^{h^\star_{\kappa-1}}(y^\star_{\kappa-1}),
&\quad f(y^\star_\kappa),
&\quad f^2(y^\star_\kappa),
&\quad \ldots,
&\quad f^{h^\star_\kappa-1}(y^\star_\kappa),
\vspace{3pt}\\
y'=f^{h^\star_\kappa}(y^\star_\kappa),
\end{array}
\end{equation}
where, apart possibly from $y^\star_0$ and $y'$, the terms do not belong to~$\IIII_0$.

Note that the two sequences \eqref{eq6.16} and \eqref{eq6.17} have the same endpoint~$y'$.
Working backward along \eqref{eq6.16} and \eqref{eq6.17} one term at a time,
we may reach $y_0$ or $y^\star_0$ first, or at the same step.

If this process reaches $y_0$ first, then $y_0$ must be one of the terms in \eqref{eq6.17} that is different from $y^\star_0$ and~$y'$,
so that $y_0\not\in\IIII_0$, a contradiction.

If this process reaches $y^\star_0$ first, then $y^\star_0$ must be one of the terms in \eqref{eq6.16} that is different from $y_0$ and~$y'$,
so that $y^\star_0\not\in\IIII_0$, a contradiction.

If the process reaches $y_0$ and $y^\star_0$ at the same step, then clearly $y^\star_0=y_0$.
Then successively we have $y^\star_1=y_1$, $y^\star_2=y_2$, and so on, and eventually $y'=y_{\kappa+1}$.
But $y'\in\IIII_0$, and $y_{\kappa+1}\in\IIII_{\kappa+1}=\phi(\IIII_\kappa)\setminus\IIII_0$ does not belong to~$\IIII_0$,
thus leading to a contradiction also.

This completes the proof.
\end{proof}

We complete this section by giving a complete answer to the question of the finiteness of the extension process.

\begin{theorem}\label{thm7}
Consider dissipative flow of slope~$\alpha$, where $\alpha>0$ is irrational, on a system $\PPP$,
where a finite polysquare translation surface has beem modified with the inclusion of a one-sided barrier $B$ on the common vertical edge
of two neighbouring atomic squares, in the form of a union of finitely many vertical intervals,
and where a set $A$ corresponding to $B$ has been chosen on a different vertical edge on the same horizontal street that contains~$B$.
Then the recurrent set $\RRR(\PPP;\alpha)$ and the transient set $\WWW(\PPP;\alpha)$ can be constructed by an extension process
on the no-go zone $\NNN(\PPP;\alpha)$.
Furthermore, this process is finite, and both sets $\RRR(\PPP;\alpha)$ and $\WWW(\PPP;\alpha)$ are finite unions of polygons
with boundary edges that are vertical or of slope~$\alpha$.
\end{theorem}

We first make some preliminary discussion where we introduce the new ideas.

As before, we make use of functions $f$, $\pi$ and $g$ satisfying \eqref{eq6.3}--\eqref{eq6.6} first introduced earlier in this section,
where we use the interval $[0,s)$ to describe the left vertical edges of the $s$ atomic squares of the finite polysquare translation surface,
with $[i,i+1)$ representing the left vertical edge of the $i$-th atomic square, where the index $i\in\MMMM=\{0,1,\ldots,s-1\}$.

The one-sided barrier $B$ on a common vertical edge of two atomic squares is a union of finitely many intervals, so that
\begin{displaymath}
B\subset[i,i+1)\subset[0,s)
\end{displaymath}
for some $i\in\MMMM=\{0,1,\ldots,s-1\}$, and the boundary $\partial B$ is a finite set.
Denote the set of \textit{splitting} points by
\textcolor{white}{xxxxxxxxxxxxxxxxxxxxxxxxxxxxxx}
\begin{displaymath}
\QQQ=\pi(\partial B)\cup\{0\}.
\end{displaymath}

To avoid these splitting points, we say that an interval $I=(a,b)\subset[0,1)$ is $k$-stable if
\textcolor{white}{xxxxxxxxxxxxxxxxxxxxxxxxxxxxxx}
\begin{displaymath}
g^j(I)\cap Q=\emptyset
\quad
\mbox{for every $j=0,1,\ldots,k$}.
\end{displaymath}
It is easy to check that the following hold:

(S1)
If $I'\subset I''$, then $I'$ is $k$-stable if $I''$ is $k$-stable.

(S2)
If $I$ is $k$-stable, then $I$ is $j$-stable for every $j=0,1,\ldots,k$.

(S3)
If $k\ge2$ and $I$ is $k$-stable, then $g(I)$ is $(k-1)$-stable.

(S4)
If $I=(a,b)$ is not $k$-stable, then there exists a finite partition
\begin{displaymath}
a=x_0<x_1<\ldots<x_n=b
\end{displaymath}
of $I$ such that $(x_{\ell-1},x_\ell)$ is $k$-stable for every $\ell=1,\ldots,n$.

(S5)
If $I$ is $1$-stable, then there exists a mapping $\sigma_I:\MMMM\to\MMMM$ such that
\begin{equation}\label{eq6.18}
f(I+i)=g(I)+\sigma_I(i),
\quad
i\in\MMMM.
\end{equation}
To see this, note that if $I$ is $1$-stable, then it follows from \eqref{eq6.6} that
\begin{displaymath}
\pi(f(I+i))=g(\pi(I+i))=g(\pi(I))=g(I)
\end{displaymath}
does not split, and so it follows that $f(I+i)\subset[i',i'+1)$ for some $i'\in\MMMM$.
We now simply take $i'=\sigma_I(i)$.

The assumption that $I$ is $k$-stable leads to a sequence
\begin{equation}\label{eq6.19}
I\longrightarrow g(I)\longrightarrow g^2(I)\longrightarrow\ldots\longrightarrow g^k(I)
\end{equation}
of splitting-free intervals.
In view of (S2), (S3) and (S5), this gives rise to a sequence
\begin{displaymath}
\MMMM
\mathop{\longrightarrow}^{\sigma_I}
\MMMM
\mathop{\makebox[0.85cm]{\rightarrowfill}}^{\sigma_{g(I)}}
\MMMM
\longrightarrow
\ldots
\longrightarrow
\MMMM
\mathop{\makebox[1.3cm]{\rightarrowfill}}^{\sigma_{g^{k-1}(I)}}
\MMMM
\end{displaymath}
of mappings
\textcolor{white}{xxxxxxxxxxxxxxxxxxxxxxxxxxxxxx}
\begin{displaymath}
\sigma_{g^j(I)}:\MMMM\to\MMMM,
\quad
j=0,1,\ldots,k-1.
\end{displaymath}
Denote their composition mapping $\sigma_{k,I}:\MMMM\to\MMMM$ by
\begin{displaymath}
\sigma_{k,I}=\sigma_{g^{k-1}(I)}\circ\ldots\circ\sigma_{g(I)}\circ\sigma_I.
\end{displaymath}
Then corresponding to (S5) and \eqref{eq6.18}, we have
\begin{equation}\label{eq6.20}
f^k(I+i)=g^k(I)+\sigma_{k,I}(i),
\quad
i\in\MMMM.
\end{equation}

\begin{remark}
Note that \eqref{eq6.18} and \eqref{eq6.20} represent not only equality of sets but crucially also mappings of intervals.
\end{remark}

Next, to avoid self intersection of the sets in the sequence \eqref{eq6.19}, we say that
an interval $I=(a,b)\subset[0,1)$ is $k$-clear if $I$ is $k$-stable and the sets
\begin{displaymath}
I,g(I),g^2(I),\ldots,g^k(I)
\end{displaymath}
are pairwise disjoint.
It is easy to check that the following hold:

(C1)
If $I$ is $k$-stable, then there exists a $k$-clear subinterval $I'\subset I$.
To see this, take $x\in I$ to be the midpoint of~$I$, let
\begin{displaymath}
r=\min_{\substack{{j_1,j_2\in\{0,1,\ldots,k\}}\\{j_1\ne j_2}}}\vert g^{j_1}(x)-g^{j_2}(x)\vert,
\end{displaymath}
and consider the interval $(x-r/2,x+r/2)$.
We now simply take a sufficiently small $r'<r/2$ to ensure that $I'=(x-r',x+r')\subset I$.
Here the choice of $x$ is justified by the Remark above, and the value of $r$ is independent of the choice of $x\in I$.

(C2)
If $I$ is $k$-clear and $i'\not\in\sigma_{k,I}(\MMMM)$, then for every $i\in\MMMM$ and integer $j\ge0$,
\begin{equation}\label{eq6.21}
f^j(I+i)\cap(g^k(I)+i')=\emptyset.
\end{equation}
To see this, suppose on the contrary that
\begin{displaymath}
x_1\in f^j(I+i)\cap(g^k(I)+i'),
\quad\mbox{so that}\quad
\pi(x_1)\in g^j(I)\cap g^k(I).
\end{displaymath}
This means that $g^j(I)\cap g^k(I)\ne\emptyset$, so that we must have $j\ge k$.
There exists $x_2\in I+i$ such that $x_1=f^j(x_2)$.
Then the number $x_3=f^{j-k}(x_2)$ satisfies
\begin{displaymath}
f^k(x_3)=f^j(x_2)=x_1\in g^k(I)+i',
\end{displaymath}
and it follows from \eqref{eq6.20} that $i'\in\sigma_{k,I}(\MMMM)$, a contradiction.

(C3)
If $I$ is $k$-clear, $i'\in\sigma_{k,I}(\MMMM)$ and $i''\not\in\sigma_{k,I}(\MMMM)$, then for every integer $j\ge0$,
\begin{displaymath}
f^j(g^k(I)+i')\cap(g^k(I)+i'')=\emptyset.
\end{displaymath}
To see this, suppose on the contrary that
\begin{displaymath}
x_1\in f^j(g^k(I)+i')\cap(g^k(I)+i'').
\end{displaymath}
Then there exists $x_2\in g^k(I)+i'$ such that $f^j(x_2)=x_1$.
Meanwhile, it follows from \eqref{eq6.20} that there exists $i\in\MMMM$ such that
\begin{displaymath}
f^k(I+i)=g^k(I)+i',
\end{displaymath}
so there exists $x_3\in I+i$ such that $f^k(x_3)=x_2$.
Thus $x_1=f^{j+k}(x_3)$, so that $f^{j+k}(I+i)\cap(g^k(I)+i'')\ne\emptyset$, contradicting \eqref{eq6.21} and (C2).

The proof of Theorem~\ref{thm7} is based on the following critical intermediate result.

\begin{lemma}\label{lem64}
If $I=(a,b)\subset[0,1)$ is $k$-clear, then one of the following holds:

\emph{(i)}
The set
\textcolor{white}{xxxxxxxxxxxxxxxxxxxxxxxxxxxxxx}
\begin{equation}\label{eq6.22}
\bigcup_{i\in\sigma_{k,I}(\MMMM)}(g^k(I)+i)
\end{equation}
generates a recurrent region in a finite number of steps.

\emph{(ii)}
There exists a subinterval $I^*\subset I$ and an integer $k^*>k$ such that $I^*$ is $k^*$-clear and
\textcolor{white}{xxxxxxxxxxxxxxxxxxxxxxxxxxxxxx}
\begin{displaymath}
\vert\sigma_{k^*,I^*}(\MMMM)\vert\le\vert\sigma_{k,I}(\MMMM)\vert-1.
\end{displaymath}
\end{lemma}

Here a \textit{recurrent region} is a set of recurrent points that is invariant under the dissipative flow.

\begin{remark}
Suppose that a recurrent region $\RRR$ of a dissipative system $\PPP$ has non-trivial intersection with every other recurrent region of~$\PPP$.
Then apart possibly from exceptional sets of measure~$0$, $\RRR$ contains every other recurrent region of~$\PPP$.
To see this, note that if $\RRR'$ is another recurrent region of~$\PPP$, then so is $\RRR'\setminus\RRR$.
If $\RRR'\setminus\RRR$ is non-trivial, then $\RRR$ intersects non-trivially with $\RRR'\setminus\RRR$, clearly a contradiction.
Hence apart possibly from exceptional sets of measure~$0$, we must have $\RRR'\subset\RRR$.

Here we have used the fact that if $\RRR$ and $\RRR'$ are recurrent regions of~$\PPP$, then so are the region $\RRR\cup\RRR'$
and its three pairwise disjoint constituent parts $\RRR\cap\RRR'$, $\RRR\setminus\RRR'$ and $\RRR'\setminus\RRR$.
To see this, note that it is straightforward that $\RRR\cup\RRR'$ and $\RRR\cap\RRR'$ are recurrent regions of~$\PPP$.
Clearly $\RRR'\setminus\RRR\subset\RRR'$, so no point of $\RRR'\setminus\RRR$ can move into $\RRR\setminus\RRR'$ under the dissipative flow.
To show that $\RRR'\setminus\RRR$ is a recurrent region,
suppose on the contrary that $\RRR'\setminus\RRR$ is not invariant under the dissipative flow,
and a positive proportion of it can move into $\RRR\cap\RRR'$ under the flow.
Since the dissipative flow is measure preserving, it follows that a positive proportion of $\RRR\cap\RRR'$ moves outside itself under the flow,
contradicting that $\RRR\cap\RRR'$ is a recurrent region.
\end{remark}

\begin{proof}[Proof of Theorem~\ref{thm7}]
For any integer~$k$, in view of (S4) and (C1), we can find a $k$-clear interval $I$ and determine the set $\sigma_{k,I}(\MMMM)$.
From Lemma~\ref{lem64}, we can keep shrinking until we obtain a minimal value of $\vert\sigma_{k,I}(\MMMM)\vert$.
This shrinking process must be finite, as $\vert\sigma_{k,I}(\MMMM)\vert\le s$ is a trivial bound.
Thus the union
\begin{displaymath}
\IIII=\bigcup_{i\in\sigma_{k,I}(\MMMM)}(g^k(I)+i)
\end{displaymath}
of subintervals of $[0,s)$ generates a recurrent region $\RRR$ of the system $\PPP$ in a finite number of steps.
To show that $\RRR$ is the whole recurrent set $\RRR(\PPP;\alpha)$,
it suffices to show that it intersects non-trivially with every other recurrent region $\RRR'$ of~$\PPP$.

Let $\RRR'$ be any recurrent region generated by $\IIII'\subset[0,s)$.
Then $\pi(\IIII')=[0,1)$, so that there exists $i^*\in\MMMM$ such that the set $\IIII'\cap(I+i^*)$ has positive measure.
The assumption that $\RRR'$ is a recurrent region implies that the set $\IIII'\cap f^k(I+i^*)$ also has positive measure.
Thus the intersection of $\RRR'\cap\RRR$ has positive measure,
and this completes the proof.
\end{proof}

It remains to establish the critical intermediate result.

\begin{proof}[Proof of Lemma~\ref{lem64}]
For every $x\in g^k(I)$, let
\begin{displaymath}
n(x)=n_{g^k(I)}(x)=\min\{n\ge1:g^n(x)\in g^k(I)\}.
\end{displaymath}
Then it follows from Lemma~\ref{lem52} due to Kac applied to the set $g^k(I)$ that there exists a common threshold $N^*(g^k(I))$ such that
\begin{displaymath}
n(x)=n_{g^k(I)}(x)\le N^*(g^k(I))
\quad
\mbox{for all $x\in g^k(I)$}.
\end{displaymath}
This and (S4) give rise to a partition
\begin{displaymath}
a_1=x_0<x_1<\ldots<x_n=b_1
\end{displaymath}
of the interval $g^k(I)=(a_1,b_1)$ such that for every $\ell=1,\ldots,n$,

$\circ$
the subinterval $D_\ell=(x_{\ell-1},x_\ell)$ is $N^*(g^k(I))$-stable; and

$\circ$
the value $n_{g^k(I)}(x)$ is constant in $D_\ell=(x_{\ell-1},x_\ell)$, with common value~$k_\ell$,
so that $D_\ell=(x_{\ell-1},x_\ell)$ is $k_\ell$-stable,
and we can consider the mapping $\sigma_{k_\ell,D_\ell}$.

We can consider a corresponding partition on the original interval $I=(a,b)$ into subintervals $P_\ell$, $\ell=1,\ldots,n$,
such that $g^k(P_\ell)=D_\ell$.

Let
\textcolor{white}{xxxxxxxxxxxxxxxxxxxxxxxxxxxxxx}
\begin{displaymath}
\MMMM_1=\sigma_{k,I}(\MMMM)
\quad\mbox{and}\quad
\MMMM_2=\MMMM\setminus\MMMM_1.
\end{displaymath}
It follows from (C3) that if $i_1\in\MMMM_1$ and $i_2\in\MMMM_2$, then for every integer $j\ge0$,
\begin{equation}\label{eq6.23}
f^j(g^k(I)+i_1)\cap(g^k(I)+i_2)=\emptyset.
\end{equation}

For every $\ell=1,\ldots,n$, we have $D_\ell\subset g^k(I)$ and $g^{k_\ell}(D_\ell)\subset g^k(I)$,
so that choosing $j=k_\ell$, it follows from \eqref{eq6.23} that
\begin{displaymath}
f^{k_\ell}(D_\ell+i_1)\cap(g^{k_\ell}(D_\ell)+i_2)=\emptyset,
\end{displaymath}
so that $\sigma_{k_\ell,D_\ell}(i_1)\ne i_2$, and so
\textcolor{white}{xxxxxxxxxxxxxxxxxxxxxxxxxxxxxx}
\begin{displaymath}
\sigma_{k_\ell,D_\ell}(\MMMM_1)\cap\MMMM_2=\emptyset.
\end{displaymath}
Let
\textcolor{white}{xxxxxxxxxxxxxxxxxxxxxxxxxxxxxx}
\begin{align}
\MMMM_{3,\ell}
&
=\sigma_{k_\ell+k,P_\ell}(\MMMM)
=\sigma_{k_\ell,D_\ell}(\sigma_{k,P_\ell}(\MMMM))
\nonumber
\\
&
\subset\sigma_{k_\ell,D_\ell}(\sigma_{k,I}(\MMMM))
=\sigma_{k_\ell,D_\ell}(\MMMM_1)
\subset\MMMM_1.
\nonumber
\end{align}
Then there are two possibilities:

(i)
Suppose that $\MMMM_{3,\ell}=\MMMM_1$ for every $\ell=1,\ldots,n$.
Note that we have arrived at the set \eqref{eq6.22} in a finite number of steps,
and the function $f$ leads to an interval exchange transformation from
\begin{displaymath}
\bigcup_{i\in\MMMM_1}(g^k(I)+i)
=\bigcup_{i\in\MMMM_1}\bigcup_{\ell=1}^n\,(D_\ell+i)
\end{displaymath}
to
\textcolor{white}{xxxxxxxxxxxxxxxxxxxxxxxxxxxxxx}
\begin{align}
\bigcup_{i\in\MMMM_1}\bigcup_{\ell=1}^nf^{k_\ell}(D_\ell+i)
&
=\bigcup_{i\in\MMMM_1}\bigcup_{\ell=1}^n\left(g^{k_\ell}(D_\ell)+\sigma_{k_\ell,D_\ell}(i)\right)
=\bigcup_{i\in\MMMM_1}\bigcup_{\ell=1}^n\left(g^{k_\ell}(D_\ell)+i\right)
\nonumber
\\
&
=\bigcup_{i\in\MMMM_1}\bigcup_{\ell=1}^n\left(D_\ell+i\right)
=\bigcup_{i\in\MMMM_1}(g^k(I)+i).
\nonumber
\end{align}
This gives case (i) in the lemma.

(ii)
Suppose that there exists $\ell=1,\ldots,n$ such that $\MMMM_{3,\ell}\ne\MMMM_1$.
Let $k^*=k_\ell+k$.
Then
\textcolor{white}{xxxxxxxxxxxxxxxxxxxxxxxxxxxxxx}
\begin{displaymath}
\vert\sigma_{k^*,P_\ell}(\MMMM)\vert\le\vert\MMMM_1\vert-1,
\end{displaymath}
and $P_\ell$ is $k^*$-stable.
It then follows from (C1) that there exists a subinterval $I^*\subset P_\ell$ which is $k^*$-clear,
and it is easy to see that
\begin{displaymath}
\vert\sigma_{k^*,I^*}(\MMMM)\vert
\le\vert\sigma_{k^*,P_\ell}(\MMMM)\vert
\le\vert\MMMM_1\vert-1
=\vert\sigma_{k,I}(\MMMM)\vert-1.
\end{displaymath}
This gives case (ii) in the lemma.
\end{proof}

%
%

\section{Further comments}\label{sec7}

Here we have studied dynamical systems where the underlying domain is a finite polysquare translation surface,
and the one-sided barriers are located on the vertical sides of some atomic squares.
For such \textit{nice} systems, we have shown that the attractor is a union of finitely many polygons,
and that this set can be determined by a finite process.

If we leave the class of finite polysquare translation surfaces, the structure of the attractor can be significantly different.
The following example, considered by Boshernitzan and Kornfeld~\cite{BK95} in 1995, provides an excellent illustration.

Let $\alpha=0.311\ldots$ be the unique root of the polynomial $p(x)=x^3-x^2-3x+1$ in the interval $(0,1)$.
Consider dissipative flow from left to right with slope $\alpha$ on the unit torus $[0,1)^2$ modified by the inclusion of
a one-sided barrier as shown in Figure~39.
When the flow encounters the bold barrier, it continues in the same direction from the corresponding point on the left edge.
For convenience, assume that the top barrier includes the bottom endpoint and the bottom barrier excludes the top endpoint.

\begin{displaymath}
\begin{array}{c}
\includegraphics[scale=0.8]{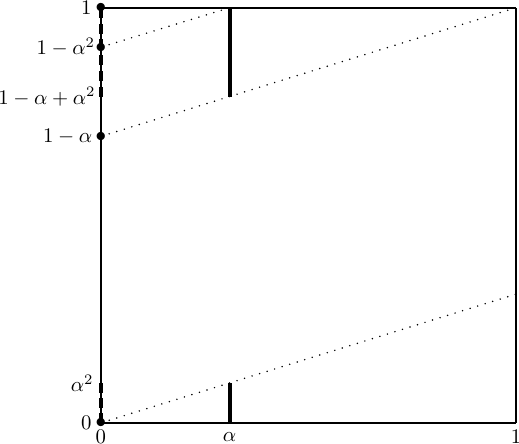}
\\
\mbox{Figure 39: a dissipative system with $1$-direction flow}
\end{array}
\end{displaymath}

This dissipative flow is described by Boshernitzan and Kornfeld in terms of an \textit{interval translation mapping}
which is a generalization of the concept of an interval exchange transformation.
Here the interval $[0,1)$ denotes the left vertical edge.
Then the interval translation mapping is given by $\T:[0,1)\to[0,1)$, where
\begin{equation}\label{eq7.1}
\T(x)=\left\{\begin{array}{ll}
x+\alpha,&\mbox{if $x\in[0,1-\alpha)$},\\
x+\alpha^2,&\mbox{if $x\in[1-\alpha,1-\alpha^2)$},\\
x+\alpha^2-1,&\mbox{if $x\in[1-\alpha^2,1)$}.
\end{array}\right.
\end{equation}
The flow that starts from any point $x$ on the left vertical edge returns to the left vertical edge at the point $\T(x)$.

Let $I_1=[1-\alpha,1)$.
To study the first return map $\T^*_1:I_1\to I_1$, we consider the partition
\textcolor{white}{xxxxxxxxxxxxxxxxxxxxxxxxxxxxxx}
\begin{displaymath}
I_1=[1-\alpha,1-\alpha^2)\cup[1-\alpha^2,1-\alpha^3)\cup[1-\alpha^3,1),
\end{displaymath}
in the same ratio as the partition $[0,1)=[0,1-\alpha)\cup[1-\alpha,1-\alpha^2)\cup[1-\alpha^2,1)$.
If $y\in[1-\alpha,1-\alpha^2)$, then
\begin{displaymath}
\T(y)=y+\alpha^2\in[1-\alpha+\alpha^2,1)\subset[1-\alpha,1)=I_1,
\end{displaymath}
so that $\T^*_1(y)=y+\alpha^2$.
If $y\in[1-\alpha^2,1-\alpha^3)\subset[1-\alpha^2,1)$, then
\begin{align}
\T(y)
&
=y+\alpha^2-1\in[0,\alpha^2-\alpha^3)=[0,1-3\alpha)\subset[0,1-\alpha),
\nonumber
\\
\T^2(y)
&
=y+\alpha^2-1+\alpha\in[\alpha,1-2\alpha)\subset[0,1-\alpha),
\nonumber
\\
\T^3(y)
&
=y+\alpha^2-1+2\alpha\in[2\alpha,1-\alpha)\subset[0,1-\alpha),
\nonumber
\\
\T^4(y)
&
=y+\alpha^2-1+3\alpha=y+\alpha^3\in[3\alpha,1)\subset[1-\alpha,1)=I_1,
\nonumber
\end{align}
so that $\T^*_1(y)=y+\alpha^3$.
If $y\in[1-\alpha^3,1)\subset[1-\alpha^2,1)$, then
\begin{align}
\T(y)
&
=y+\alpha^2-1\in[\alpha^2-\alpha^3,\alpha^2)=[1-3\alpha,\alpha^2)\subset[0,1-\alpha),
\nonumber
\\
\T^2(y)
&
=y+\alpha^2-1+\alpha\in[1-2\alpha,\alpha^2+\alpha)\subset[0,1-\alpha),
\nonumber
\\
\T^3(y)
&
=y+\alpha^2-1+2\alpha=y+\alpha^3-\alpha\in[1-\alpha,\alpha^2+2\alpha)\subset[1-\alpha,1)=I_1,
\nonumber
\end{align}
so that $\T^*_1(y)=y+\alpha^3-\alpha$.
Summarizing, we have
\begin{equation}\label{eq7.2}
\T^*_1(y)=\left\{\begin{array}{ll}
y+\alpha^2,&\mbox{if $y\in[1-\alpha,1-\alpha^2)$},\\
y+\alpha^3,&\mbox{if $y\in[1-\alpha^2,1-\alpha^3)$},\\
y+\alpha^3-\alpha,&\mbox{if $y\in[1-\alpha^3,1)$}.
\end{array}\right.
\end{equation}
Consider next the linear mapping from $[0,1)$ to $[1-\alpha,1)$, given by
\begin{equation}\label{eq7.3}
y=1-\alpha+x\alpha.
\end{equation}
Using \eqref{eq7.1}, we see that
\begin{displaymath}
y(\T(x))=\left\{\begin{array}{ll}
1-\alpha+(x+\alpha)\alpha=y+\alpha^2,&\mbox{if $x\in[0,1-\alpha)$},\\
1-\alpha+(x+\alpha^2)\alpha=y+\alpha^3,&\mbox{if $x\in[1-\alpha,1-\alpha^2)$},\\
1-\alpha+(x+\alpha^2-1)\alpha=y+\alpha^3-\alpha,&\mbox{if $x\in[1-\alpha^2,1)$}.
\end{array}\right.
\end{displaymath}
Comparing this with \eqref{eq7.2}, we conclude that the linear mapping \eqref{eq7.3}
gives rise to an isomorphism between $\T:[0,1)\to[0,1)$ and $\T^*_1:[1-\alpha,1)\to[1-\alpha,1)$.
In other words, the interval translation mapping $\T$ acts on the subinterval $[1-\alpha,1)$
in precisely the same way as on the interval $[0,1)$.
This can be interpreted as \textit{self similarity in a smaller scale}.

Next, let $I_2=[0,\alpha^2)$.
To study the first return map $\T^*_2:I_2\to I_2$, we consider the partition
\textcolor{white}{xxxxxxxxxxxxxxxxxxxxxxxxxxxxxx}
\begin{displaymath}
I_2=[0,\alpha^2-\alpha^3)\cup[\alpha^2-\alpha^3,\alpha^2-\alpha^4)\cup[\alpha^2-\alpha^4,\alpha^2).
\end{displaymath}
We can show that
\begin{equation}\label{eq7.4}
\T^*_2(y)=\left\{\begin{array}{ll}
y+\alpha^3,&\mbox{if $y\in[0,\alpha^2-\alpha^3)$},\\
y+\alpha^4,&\mbox{if $y\in[\alpha^2-\alpha^3,\alpha^2-\alpha^4)$},\\
y+\alpha^4-\alpha^2,&\mbox{if $y\in[\alpha^2-\alpha^4,\alpha^2)$}.
\end{array}\right.
\end{equation}
Furthermore, consider the linear mapping from $[0,1)$ to $[0,\alpha^2)$, given by
\begin{equation}\label{eq7.5}
y=\alpha^2x.
\end{equation}
Using \eqref{eq7.1}, we see that
\begin{displaymath}
y(\T(x))=\left\{\begin{array}{ll}
\alpha^2(x+\alpha)=y+\alpha^3,&\mbox{if $x\in[0,1-\alpha)$},\\
\alpha^2(x+\alpha^2)=y+\alpha^4,&\mbox{if $x\in[1-\alpha,1-\alpha^2)$},\\
\alpha^2(x+\alpha_2-1)=y+\alpha^4-\alpha^2,&\mbox{if $x\in[1-\alpha^2,1)$}.
\end{array}\right.
\end{displaymath}
Comparing this with \eqref{eq7.4}, we conclude that the linear mapping \eqref{eq7.5}
gives rise to an isomorphism between $\T:[0,1)\to[0,1)$ and $\T^*_2:[0,\alpha^2)\to[0,\alpha^2)$.
In other words, the interval translation mapping $\T$ acts on the subinterval $[0,\alpha^2)$
in precisely the same way as on the interval $[0,1)$.

Elaborating on this idea, Boshernitzan and Kornfeld can show that the attractor in this example
exhibits Cantor type self similarity and is an uncountable nowhere dense compact set
with measure zero.

Skripchenko and Troubetskoy~\cite{ST15a,ST15b} have made a study of these systems after
the initial work of Boshernitzan and Kornfeld.
They can show that by varying the position and length of the one-sided barrier and the slope of the flow,
one can construct uncountably many similar systems where the Hausdorff dimension of the closure
of the attractor is zero.
They also study the question of entropy which is beyond the scope of our study here.

%
%

\end{document}